\documentclass[aos,preprint]{imsart}

\RequirePackage{amsthm,amsmath,amsfonts,amssymb}
\usepackage{multirow}

\RequirePackage[authoryear]{natbib} 
\RequirePackage{graph icx}

\startlocaldefs
\theoremstyle{plain}
\newtheorem{theorem}{Theorem}
\newtheorem{proposition}{Proposition}

\theoremstyle{remark}
\newtheorem{example}{Example}
\newtheorem{remark}{Remark}
\newtheorem{definition}{Definition}

\newcommand{\Ex}{\mathbf{E}}
\newcommand{\E}{\mathbf{E}}
\newcommand{\p}{\mathbf{P}}
\newcommand{\V}{\mathbf{Var}}
\newcommand{\R}{\mathbb{R}}
\newcommand{\mat}{\mathbb}
\newcommand{\A}{\mathcal{A}}
\newcommand{\vect}{\operatorname{vec}}

\endlocaldefs

\begin{document}

\begin{frontmatter}
\title{Design of $c$-Optimal Experiments for High dimensional Linear Models}
\runtitle{High-dimensional $c$-Optimal Designs}

\begin{aug}
\author[]{\fnms{Hamid} \snm{Eftekhari}\ead[label=e1]{hamidef@umich.edu}},
\author[]{\fnms{Moulinath} \snm{Banerjee}\ead[label=e2]{moulib@umich.edu}}\and
\author[]{\fnms{Ya'acov} \snm{Ritov}\ead[label=e3]{yritov@umich.edu}}
\address[]{Department of Statistics, University of Michigan}

\address[]{\printead{e1}, \printead{e2}, \printead{e3}}
\end{aug}

\begin{abstract}
We study random designs that minimize the asymptotic variance of a de-biased lasso estimator when a large pool of unlabeled data is available but measuring the corresponding responses is costly. The optimal sampling distribution arises as the solution of a semidefinite program. The improvements in efficiency that result from these optimal designs are demonstrated via simulation experiments.
\end{abstract}

\begin{keyword}[class=MSC2020]
\kwd[Primary ]{62K05}
\kwd[; secondary ]{62J07}
\end{}


\end{frontmatter}


\section{Introduction}
Optimal design of experiments is a statistical framework for improving the efficiency of experiments with wide-ranging applications in science and industry. In the case of regression, given the covariate data ($x_i$'s) and a fixed sample size ($n$), the goal is to measure the responses ($y_i$'s) only for those $x_i$'s that result in the most accurate estimate of the relationship between the covariates and the response. This can result in significant savings in experiment time and/or monetary expenses associated with the experiment. 

The relationship between the response ($y$) and covariates ($x$) is often modeled as a linear model:
\begin{align*}
y = \langle x, \beta^\star \rangle + \varepsilon,
\end{align*}
where $\beta^\star \in \mathbf{R}^p$ is the true regression parameter and $\varepsilon$ denotes a mean-zero noise term. The target of estimation is a linear combination $\langle c, \beta^\star \rangle$ of the regression parameter $\beta^\star$. In the classical framework, a $c$-optimal design is a design (i.e. a distribution on the covariate data $(x_i)_i$) that minimizes the variance of the ordinary least squares (OLS) estimate of $\langle c, \beta^\star \rangle$.

The optimal design problem has been thoroughly investigated in the case of low-dimensional (fixed dimension, increasing number of observations) linear regression. Various optimality criteria, usually defined as functionals of the information matrix, have been considered and their relative merits are studied in the literature. \cite{pukelsheim2006optimal} provides an excellent introduction to the theory of optimal design. More relevant to our work is the literature on $c$-optimal design, pioneered by the seminal work of \cite{elfving1952optimum}, who provided an elegant geometric solution for the (approximate) $c$-optimal design problem which can also be cast as a linear program admitting an efficient solution. Variations and extensions of Elfving's theorem have since been studied. \cite{chernoff1953locally} generalized Elfving's result beyond the regression framework by introducing a local and asymptotic notion of optimality based on the information matrix. \cite{sagnol2011computing} studied $c$-optimality for multi-response regression and showed that it can be formulated as a second-order cone programming (SOCP) problem. \cite{dette1993elfving} considered a more general version of the $c$-optimality criterion that is robust across classes of models and proved a generalization of Elfving's result for this criterion. $c$-optimal designs have also been studied for models with correlated \citep{rodriguez2017computation} and heteroscedastic \citep{dette2009geometric} error terms.

In the high-dimensional regime, where the dimension of covariates is larger than the sample size, consistent estimation and inference are hopeless without further structural assumptions, the most well-studied being sparsity of $\beta^\star$. For estimation in sparse linear models, the $\ell_1$-regularized least squares (also known as the lasso) estimator has been proposed \citep{tibshirani1996regression} and shown to enjoy strong theoretical guarantees \citep{bickel2009simultaneous}. More recently, the de-biasing techniques have been introduced for correcting the bias of the lasso estimator, thereby providing $\sqrt{n}$-consistent and asymptotically normal estimates of the coordinates of $\beta^\star$ and allowing the construction of confidence intervals for its finite-dimensional coordinates \citep{vandegeer2014,zhang2014confidence,javanmard2014confidence}. Variations of the de-biased lasso estimator have also been devleoped for estimating linear combinations $\langle c, \beta^\star \rangle$ of $\beta^\star$ \citep{cai2017confidence,cai2019individualized}.

However, the optimal design problem for inference in high-dimensional linear models {\it has received scant attention}. \cite{seeger2008bayesian} studied sequential design for maximizing information gain in sparse linear models in a Bayesian framework. \cite{ravi2016experimental} introduced $D$-optimal designs for sparse linear models that combine the classical $D$-optimiality criterion applied to a perturbation of the information matrix and a term that penalizes the $l_2$ distance between the top eigenvectors of the full and reduced Hessians. \cite{Deng_thelasso} used nearly orthogonal Latin hypercube designs to improve variable selection using the lasso. More recently, \cite{huang2020optimal} extended Keifer's notion of $\Phi_l$-optimality to one applicable for the covariance matrix of the de-biased lasso and developed an algorithm for finding local optima of the resulting non-convex optimization problem.

\subsection{Contributions}\label{subsec:contributions}
Targeted inference on pre-selected parameters of high-dimensional models is of great interest in a variety of applications but has not been addressed in the existing experimental design literature. Our work bridges this critical gap between classical optimal design theory and frequentist inference in high-dimensional linear models. To our knowledge, this is the first rigorous result on this topic. Our contributions are articulated below:
\begin{enumerate}
\item We extend the standard convergence guarantees for the de-biased lasso beyond sub-Gaussian designs to random vectors with bounded entries using Poisson sampling and a novel proof technique. In the extant literature on the de-biased lasso, it is typically assumed that the rows of the design matrix are sub-Gaussian random vectors. While this assumption is convenient for concentration arguments, it is too restrictive for random vectors taking values in finite sets as are commonly encountered in the design problem. To appreciate that this is more than a mere technicality, see for example \citep[Exercise 3.4.5]{vershynin2018high} showing that if $X$ is an isotropic random vector supported on a finite set $S \subset \mathbf{R}^p$ and if $\| X \|_{\psi_2} \leq M$ for a bound $M$ not depending on $p$, then we must have $|S| > e^{cp}$ for a constant $c > 0$. In words, finitely-supported isotropic sub-Gaussian random vectors with bounded sub-Gaussian norms have supports that are \textit{exponentially large} in the dimension of the vector. Clearly this rules out many interesting design problems.

\item We introduce the notion of a constrained $c$-optimal design that is a natural extension of the classical $c$-optimality criterion suitable for the high-dimensional setting. In the low-dimensional regime, the OLS estimate is unbiased and the $c$-optimality criterion solely focuses on minimizing the variance of the OLS estimate of $\langle c, \beta^\star \rangle$. In the high-dimensional setting, the de-biased lasso estimate is not unbiased for finite samples (even though the bias is asymptotically $o(1/\sqrt{n})$ assuming the RE condition), and one needs to control the bias when minimizing the variance. This naturally leads to semidefinite constraints on the covariance matrix of the design distribution. Furthermore, we illustrate through an example that the de-biased lasso can have an arbitrarily large bias (with non-vanishing probability) if one minimizes the $c$-optimality criterion without regard for controlling the bias.
\end{enumerate}

\subsection{Applications}

Examples of applications for $c$-optimal designs in high-dimensional linear models include:
\begin{itemize}
\item \textbf{Detection of weak signals in sparse magnetic resonance images.} Due to physical limitations, measurements in MRI are time-consuming and there is substantial interest in reducing the number of measurements while preserving the reconstruction accuracy of the images \citep{lustig2007sparse}. In this application, optimal designs allow more efficient inference on, say, the average intensity of regions of interest (ROI) of the object under study (for example, the suspected location of a brain tumor), which enables detection of weak signals with fewer measurements than a uniform design. This example is described in detail in Section \ref{sec:application}.

\item \textbf{Improving the efficiency of clinical trials.} The amount and precision of medical data on individuals has been growing at an increasing speed in the past decades. On the other hand, medical experiments continue to be prohibitively expensive. Optimal designs can be used to improve the efficiency of experiments, thus leading to significant cost-savings \citep{atkinson1996usefulness,weng2015optimizing}. Whenever high-dimensional data is available on the {\it potential} subjects for an experiment, our results can improve the efficiency of inference on the treatment effects under study. For example, suppose that the genetic profile and medical history (high-dimensional covariates of dimension $p$) of a large number ($N \geq p$) of potential subjects is available to an investigator who is interested in the effect of an expensive drug and suppose that the treatment is pre-assigned (e.g. via randomization) and encoded as the first covariate. To save costs, it is desired to to have as few samples as possible ($n \ll p$), leading to a high-dimensional estimation problem. The question of which individuals to enroll in the experiment (based on their medical profile) leads naturally to a $c$-optimal design problem, where $c = \mathrm{e}_1$.

A similar application area is in controlled experiments on the web, where users or clients are directed to different versions of a product or service in order to measure the effect of an intervention on a metric of interest such as revenue or user engagement \citep{kohavi2009controlled}. Online businesses and service providers often have large amounts of data on their users (their personal information, browsing history, preferences, purchase history, etc.) that together with the intervention of interest can be encoded into high-dimensional covariates and which can be used to decide which users are enrolled into the aforementioned experiments.

\end{itemize}

\subsection{Notation and Definitions}
For any natural number $q$ we set $[q] := \{ 0, 1, \dots, q \}$. In this work $p$ is the dimension of covariates, $N$ is the total number of available design points, and $n$ denotes the sample size for which responses have been observed, so that in general we have $n \leq p \leq N$. Lower-case $x_i \in \R^p$ is used to denote potential covariate data available for sampling; the set $(x_i)_1^N$ of all these points is called the experimental domain. Upper-case $(X_i)_1^n$ denotes a sample from $(x_i)_1^N$ for which responses $(y_i)_1^n$ have been observed. The $n \times p$ matrix with $X_i^T$ in its $i$th row is denoted by $\mat{X}$ and we let $Y = (y_1, \dots, y_n)^T$.

The standard inner product in $\R^p$ is denoted by $\langle \cdot, \cdot \rangle$, whereas the $\ell_q$-norm is defined by $\| a \|_q = \sqrt[q]{\sum_i |a_i|^q}$ for any $q \geq 1$. The $j$-th element of the standard basis of $\R^p$ is denoted by $\mathrm{e}_j$. 

For a random variable or vector $Z$ we write $\E Z$ for the expected value of $Z$, where as $\p(E)$ is used to denote the probability of an event $E$. For a random variable $Z$ the sub-Gaussian (respectively, sub-exponential) norm is defined by $\| Z \|_{\psi_q} := \inf \{ t > 0 :\, \E \exp( |Z|^q / t^q ) < 2 \}$ with $q = 2$ (respectively,  $q = 1$).

For a positive semidefinite matrix $\Sigma$, we use $\lambda_{\min}(\Sigma)$ and $\lambda_{\max}(\Sigma)$ to denote its smallest and largest eigenvalues. For two matrices $A$ and $B$, the relation $A \preccurlyeq B$ means $B - A$ is positive semidefinite. For a matrix $\Sigma$ we write $\Sigma^+$ for its pseudo-inverse. 

For two sequence $a_n, b_n$ we write $a_n \lesssim b_n$ to mean that there exists an absolute constant $C>0$ (that does not depend on $n,p,N$) such that $a_n \leq C b_n$ for all $n \geq 1$. The reverse inequality $a_n \gtrsim b_n$ means $b_n \lesssim a_n$, while $a_n \asymp b_n$ means both $a_n \lesssim b_n$ and $b_n \lesssim a_n$ hold.

\subsection{Problem Formulation}
We assume that a large pool $\mathcal{X} = (x_{i})_{i=1}^N$ of covariate data (called the experimental domain) is available for which no responses have been observed yet. Upon sampling an $x \in (x_{i})_{i=1}^N$,  the experimenter observes the corresponding response, which we assume follows a linear model
\begin{align}\label{model1}
y = \langle x, \beta^\star \rangle + \varepsilon \tag{Model 1}
\end{align}
where $\beta^\star\in \R^p$ is an unknown parameter and $\varepsilon$ is sub-Gaussian noise independent of the choice of $x$. Crucially, we consider the case where the total number of observed responses is $n < p$, resulting in a high-dimensional inference problem. 

We assume that the target of inference is a linear combination $\langle c, \beta^\star \rangle = \sum_{i=1}^p c_i \beta_i^\star$ of $\beta^\star$, where $c$ is chosen apriori by the experimenter. For example, with $c = \mathrm{e}_j$ (the $j$th standard basis element) the aim is to conduct inference on $\beta_j^\star = \mathrm{e}_j^T \beta^\star$. Another example is the contrast $c = \mathrm{e}_i - \mathrm{e}_j$ where the goal is to conduct inference on $\beta_i^\star - \beta_j^\star$, the difference between effects of the $i$th and $j$th covariates.


We assume throughout that $N \geq p$, as it is crucial for our results that the population covariance matrix is non-singular, so that the restricted eigenvalue condition can be guaranteed for sample design matrices. In the $N < p$ case, while the restricted value condition might hold for population or sample covariance matrices, it is well-known that these conditions are NP-hard to check \citep{dobriban2016regularity}, and so equally hard to enforce on designs. 

Before presenting our results for high-dimensional models, we review the $c$-optimality criterion in the low-dimensional setting.

\subsection{$c$-optimality in low dimensions}
Suppose that $n > p$, and we are given a fixed sample $(X_i, y_i)_{i=1}^n$. Assuming that $c$ belongs to $\operatorname{span}(X_1, \dots, X_n)$, by the Gauss-Markov theorem, the best linear unbiased estimate of $c^T \beta^\star$ is $c^T \hat{\beta}^{OLS}$ where the ordinary least squares (OLS) estimate of $\beta^\star$ is defined by
$\hat{\beta}^{\mathrm{OLS}} := (X^T X)^+ X^T y.$
Recall that $c^T \hat{\beta}^{\mathrm{OLS}}$ is unbiased and has variance equal to
$\sigma^2 \cdot c^T (X^T X)^+ c$.
The goal in $c$-optimal design of experiments is to choose $X_1, \dots, X_n$ (among $(x_1, \dots, x_N)$) in such a way that the above variance is minimized. Letting $N_i\in [n]:= \{ 0, 1, \dots, n \}$ be the number of times $x_i$ is repeated in the design, the exact $c$-optimality problem can be written as

\begin{align*}
	\mathbf{P0 : }\min_{(N_i)_1^N }  \quad &c^T \Sigma^+ c\\[-10pt]
	\text{s.t.}  \quad &\Sigma = \frac{1}{n} \sum_{i=1}^N N_i x_i x_i^T\\[-10pt]
	\quad &\sum_{i=1}^N N_i = n , \quad N_i \in [n].
\end{align*}
 Since this {\it exact} design problem (that is, with the integer constraints $N_i \in [n]$) turns out to be computationally hard to solve (more specifically, it is NP-complete \citep{vcerny2012two}), it is often relaxed to the so-called approximate design problem formulated as below:
\begin{align*}
	\mathbf{P1 : }\min_{w \in \R^{N}}  \quad &c^T \Sigma^+ c\\[-10pt]
	\text{s.t.}  \quad &\Sigma = \sum_{i=1}^N w_i x_i x_i^T\\[-10pt]
	\quad &\sum_{i=1}^N w_i = 1 ,\quad w \geq 0.
\end{align*}
Note that given an approximate optimal design $(w_i^\star)_1^N$, it is likely that $n w_i^\star$ is not an integer. Thus one may have to resort to rounding techniques to obtain an exact design from the approximate optimal design \citep{pukelsheim1992efficient}. Alternatively, one can use randomization where $\{X_1, \dots, X_n\}$ is an iid sample from $(x_i)_1^N$ with probabilities $(w_i^\star)_1^N$. We take the latter approach in this work since randomization also allows the derivation of high-probability bounds on the bias of the de-biased lasso. 

It follows from Elfving's theorem \citep{elfving1952optimum} that the above problem is equivalent to the following $\ell_1$ minimization problem:
\begin{align}\label{linear_program}
\mathbf{P1^\prime : }\min_{b \in \R^N} \quad & \| b \|_1 \\[-10pt]
\text{s.t.} \quad & c = \sum_{i=1}^N b_i x_i
\end{align}
If $b^\star$ is an optimal solution of $\mathbf{P1^\prime}$, then $w^\star = b^\star / \| b^\star \|_1$ is an optimal solution for $\mathbf{P1}$. It can be shown that $w^\star$ has at most $p$ non-zero entries.

\section{Debiased Inference of Parameters}
Various methods have been proposed for estimation and inference of linear functionals $\gamma := \langle c, \beta\rangle$ in the high-dimensional inference literature. \citet{cai2017confidence} provided minimax optimal rates for the lengths of confidence intervals depending on the sparsity structure of $c$. \citet{javanmard2017flexible} studied hypothesis testing for general null hypothesis of the form $\beta \in \Omega_0$ for arbitrary $\Omega_0$. \cite{cai2019individualized} studied inference of individualized treatment effects $\langle c, \beta_1 - \beta_2 \rangle$ for a two-sample problem and linear functionals $\langle c, \beta \rangle$ for the one-sample problem. The common theme among all these works is the use of debiasing procedures, where one starts with a biased estimate (typically variants of the lasso) and then uses a projection method to correct its bias. We describe a variant of the method proposed by \citet{cai2019individualized} in the following.

Suppose that $(Y_i, X_i)_{i=1}$ is an i.i.d. sample drawn according to $P$. To obtain an initial estimate for $\beta^\star$, use the Lasso estimator defined by

\begin{align}\label{lasso_def}
\hat{\beta} := \arg \min_{\beta' \in \R^p } \left\{ L_\lambda(\beta') := \frac{ \|Y -  \mat{X}\beta'\|_2^2}{2n} + \lambda \sum_{j=1}^p \widehat{W}_j |\beta_j'| \right\}.
\end{align}
where $\widehat{W}_j = \sqrt{n^{-1} \sum_{i} X_{ij}^2}$ and with the tuning parameter of order\footnote{This is a theoretical value for the tuning parameter since $\sigma_\varepsilon$ is typically unknown. In practice one can use an estimate of $\sigma_\varepsilon$ or use cross-validation.} $\lambda \asymp \sigma_\varepsilon\sqrt{\frac{ \log p}{n}}$.

In order to correct the bias of $\hat{\beta}$, \cite{cai2019individualized} propose to use the following estimator
\begin{align*}
\hat{\gamma} = \langle c, \hat{\beta} \rangle + \frac{1}{n} \hat{u}^T \mat{X}^T(Y - \mat{X} \hat{\beta}), 
\end{align*}
where
\begin{align*}
 \hat{u} := \arg \min_{u} \quad &{u^T \hat{\Sigma} u}\\
 \text{s.t.}  \quad &\| \hat{\Sigma}u - c \|_\infty \leq \| c\|_2 \lambda,\\
 			 & | c^T \hat{\Sigma}u - \|c\|_2^2 | \leq \| c\|_2^2\lambda. 
\end{align*}
The above minimization problem effectively estimates $u = \Sigma^{-1}c$ in the typical setting where $\Sigma$ is unknown. In our setting, however, the population of covariates is fully known (after determining a design) and thus one can directly use $u = \Sigma^{-1} c$ in the de-biasing procedure. Thus we will be using the de-biased estimator
\begin{align}\label{deb_lasso}
\hat{\gamma} := \langle c, \hat{\beta} \rangle + \frac{1}{n} u^T \mat{X}^T(Y - \mat{X} \hat{\beta}),
\end{align}
with $u = \Sigma^{-1}c$. 

Before we state our main theorem regarding the consistency and asymptotic normality of $\hat{\gamma}$ we describe and motivate the Poisson sampling scheme used in our work.\newline

\noindent\textbf{Poisson Sampling. }
Suppose that $w=(w_i)_1^N$ is a probability distribution on $(x_i)_1^N$. Given an {\it iid} sample $X_1, \dots, X_n$ according to probabilities $w$, define $\tilde{N}_i$ to be the number of times $x_i$ is repeated among $X_1, \dots, X_n$:
\begin{align*}
\tilde{N}_i := \left| \{ k \in \{ 1, \dots, n\} : X_k = x_i \} \right|
\end{align*}
Then random objects such as $\hat{\Sigma} = n^{-1} \sum_1^n X_i X_i^T$ that appear in our theoretical analysis can be viewed as functions of the multinomial random variables $(N_i)_1^N$ as
\begin{align*}
\hat{\Sigma} = \frac{1}{n} \sum_{i=1}^N \tilde{N}_i \, x_i x_i^T.
\end{align*}

This shifts the focus from high-dimensional random vectors $(X_i)_1^n$ to random variables $(\tilde{N}_i)_1^N$ which are more amenable to theoretical analysis. Studying the concentration properties of $\hat{\Sigma}$ therefore necessitates dealing with the $\tilde{N}_i$'s. However, the dependence among the $\tilde{N}_i$'s leads to difficulties with concentration arguments which are typically based on independence (note that under iid sampling, the distribution of $(\tilde{N}_i)_1^N$ is multinomial with probabilities $(w_i)_1^N$ and sum equal to $n$). In order to use standard concentration results based on independence, it is useful to break this dependence among the $
\tilde{N}_i$'s, which can be achieved by using Poisson sampling as follows: Let $(N_i)_1^N$ be independent random variables with $N_i \sim \text{Poisson}(n w_i)$ and set $K = \sum_{i=1}^N N_i$. A sample $X_1, \dots, X_K$ is then created where $x_j$ appears $N_j$ times in this sample. It is easy to see that:
\begin{enumerate}
    \item The total number of samples drawn is close to $n$ (relative to $n$):
    \begin{align*}
       \frac{K}{n} \rightarrow_p 1 \text{ as } n \rightarrow \infty.
    \end{align*}
    \item Conditioned on $K = k$, the distribution of $(N_i)_{i=1}^N$ is multinomial with parameters $\left(k, (w_i)_1^N  \right)$.
\end{enumerate}
Thus the sample obtained from Poisson sampling is similar to sampling with replacement according to $w$. This Poisson sampling scheme precludes the need for sub-Gaussianity of the vectors $X_i$.

The following theorem describes the asymptotic distribution of the de-biased estimator $\hat{\gamma}$ on data obtained using Poisson sampling. The theorem is similar in spirit to the results of \cite{cai2019individualized}, \cite{javanmard2017flexible}, with the main difference being that it does not assume the covariates are sub-Gaussian vectors. Recall that as discussed in subsection \ref{subsec:contributions}, this is important in the design setting because it is not clear how to enforce the sub-Gaussianity property when finding an optimal design. We bypass this problem via Poisson sampling along with a novel proof technique that does not rely on the sub-Gaussianity of the design.

\begin{theorem} \label{thm:debiasing}
Suppose that $\mat X = (X_i^T)_1^K$ is a Poisson sample according to a distribution $w^\star$ on $(x_i)_1^N$ with $\E K = n$ and that $(y_i, X_i)_1^K$ follow \ref{model1} with a $s$-sparse regression parameter $\beta^\star$. Also assume that the following conditions are satisfied:
\begin{enumerate}
\item $\|x_i\|_\infty \leq M < \infty$.
\item $\lambda_\star I \preccurlyeq \Sigma := \sum_1^N w_i^\star x_i x_i^T \preccurlyeq  \lambda^\star I$ where $0< \lambda_\star, \lambda^\star < \infty$ do not depend on $n$.
\item The noise terms $\varepsilon_j$ are iid mean-zero sub-Gaussian random variables with $\| \varepsilon_j \|_{\psi_2} \leq \sigma_\varepsilon$ and a variance $ \E \varepsilon_n^2$ that is bounded away from $0$ and $\infty$.
\item $s \log^{\frac{3}{2}}(p)= o(\sqrt{n})$.
\end{enumerate}

Then the debiased estimate (\ref{deb_lasso}) satisfies 
\begin{align*}
    \sqrt{n}(\hat{\gamma} - \gamma) = \frac{1}{\sqrt{n}} c^T \Sigma^{-1} \mat{X}^T\varepsilon + b_n
\end{align*}
where 
\begin{itemize}
    \item (Bias Bound) with probability $1 - o(1)$, we have
    \begin{align*}
        |b_n| \lesssim \frac{M \sigma_\varepsilon \sqrt{c^T \Sigma^{-1} c}}{\lambda_*} \cdot \frac{s \log^{\frac{3}{2}}(p)}{\sqrt{n}}.
    \end{align*}
    \item (Asymptotic Normality) Let $v^2 := \E\varepsilon_n^2 \cdot c^T \Sigma^{-1} \hat{\Sigma} \Sigma^{-1} c$ where $\hat{\Sigma} = n^{-1} \sum_1^K X_i X_i^T$. Then:
    \begin{align*}
        \frac{1}{v \sqrt{n}} c^T \Sigma^{-1}\mathbb{X}^T\varepsilon  \rightarrow_d N(0, 1).
    \end{align*}
    \item (Variance Approximation) The variance of the noise term can be approximated (asymptotically) by $c^T \Sigma^{-1} c$:
    \begin{align*}
    \frac{c^T \Sigma^{-1} \hat{\Sigma} \Sigma^{-1} c}{ c^T \Sigma^{-1} c} \rightarrow_p 1.
    \end{align*}
\end{itemize}
\end{theorem}

\noindent\textbf{Discussion. }
1. The first assumption imposes a uniform bound on the coordinates of all design points. Note that we can permit $M$ to depend on $n,p,N$, and thus slowly grow with $n$, provided the sparsity constraint (in the fourth assumption) is modified accordingly. For example, if the experimental domain $(x_i)_1^N$ is itself a sample from a population such that all coordinates $x_{ij}$ are sub-Gaussian with $ \|x_{ij}\|_{\psi_2} = \mathcal{O}(1)$, then it is well-known that $\max_{i,j} |x_{ij}| \lesssim \sqrt{ \log(Np) } $ with probability $1 - o(1)$. Thus in this case we could set $M \asymp \sqrt{\log(Np)}$ and replace assumption 4 with $s \log^{3/2}(p) \sqrt{\log(Np)} = o(\sqrt{n})$.

2. The second assumption in Theorem \ref{thm:debiasing} is used in our analysis to ensure that the restricted eigenvalue condition is satisfied (with a non-negligible restricted eigenvalue) for the sample design matrix with high probability, thereby guaranteeing the fast rate of convergence for the lasso estimator. Furthermore, the semidefinite nature of this constraint leads to computationally tractable optimization problems as opposed to constraints directly involving restricted eigenvalues which are themselves NP-hard to compute in general \citep{dobriban2016regularity}. Also, note that this condition implies that $c \in \mathrm{span}(\{ x_i : \, w_i^\star > 0 \})$, which is a necessary condition for unbiasedness in the low-dimensional setting.
As illustrated in the following example, in the absence of this assumption the population covariance matrix $\Sigma$ resulting from $P1$ can have a vanishing restricted eigenvalue, leading to the inconsistency of the de-biased estimator.

3. The third assumption regarding the sub-Gaussianity of the noise term is standard in the literature on inference for high-dimensional linear models and is assumed to allow concentration arguments.

4. The last assumption imposes a constraint on the sparsity of the regression parameter $\beta^\star$ in terms of $n,p$. Note that this is a stronger condition than the  ``ultra-sparsity'' condition assumed in high-dimensional inference\citep{vandegeer2014,zhang2014confidence,javanmard2014confidence}. The extra factor of $\sqrt{\log(p)}$ in this assumption is the price we pay for dispensing with the sub-Gaussianity of covariate vectors (as is typically assumed in the literatture) and only assuming uniformly bounded entries. 

\begin{example}
Suppose that $N = p = 2n$ are natural numbers and let $(x_i)_1^N$ be an orthogonal basis of $R^p$ that satisfies the following conditions:
\begin{itemize}
\item $\| x_i \|_\infty \leq M$ for some $ M > 0$ not depending on $n,p$.
\item The first entry of $x_i$ is given by
\begin{align*}
x_{i1} = \begin{cases}
0 & 1 \leq i \leq n \\
\sqrt{2} & n + 1 \leq i \leq p.
\end{cases}
\end{align*}
\item The $x_i$'s are orthogonal: $p^{-1} \sum_{i=1}^N x_i x_i^T = I_p$.
\end{itemize}
An explicit construction of such a basis is given in section \ref{Appendix}. Define $c$ to be a scaled sum of the first $n+1$ of $x_i$'s:
\begin{align*}
c = \frac{1}{\sqrt{n+1}} \sum_{i=1}^{n + 1} x_i,
\end{align*}
so that we have $\| c \|_2 = \sqrt{p}$. It follows from Elfving's Theorem that the $c$-optimal design obtained from problem $P1$ in this case is characterized by
\begin{align*}
w_i^\star = \begin{cases}
\frac{1}{n+1} &: 1 \leq i \leq n + 1, \\
0 &: n + 1 < i \leq N.
\end{cases}
\end{align*}
It is clear that $w^\star$ leads to a singular population covariance matrix $\Sigma_{w^\star}$. Under this optimal design, with probability $(1 - 1/(n+1))^n \approx 1/e$ the design point $x_{n+1}$ is not among $X_1, \dots, X_n$, so that on this event (called $E$) we have $X_{i1} = 0$ for $1 \leq i \leq n$. It is clear that in this case we can not hope to estimate $\beta_1^\star$, while $c^T \beta^\star$ depends on $\beta_1^\star$. For a simple example, take $\beta^\star = (\theta \sqrt{(n+1)/2}, 0, \dots, 0)^T$, so that $\gamma = \langle c, \beta^\star \rangle = \theta $. Assuming a $N(0,1)$ noise distribution for $\varepsilon_i$, it is clear that on $E$ we have $y_i = \varepsilon_i$. Given $E$, we can show that $\hat{\beta} = 0_p$ is a solution of the lasso problem (\ref{lasso_def}) with probability $1 - o(1)$ (with $\lambda = \sqrt{(2+\eta)\log(p)/n}$ for a small $\eta > 0$; see Section \ref{Appendix} for a proof). A natural extension of the de-biased estimator (\ref{deb_lasso}) with $\Sigma_{w^\star}^+$ substituted for $\Sigma_{w^\star}^{-1}$ is given by
\begin{align*}
\hat{\gamma} &= \langle \hat{\beta}, c \rangle + \frac{1}{n} c^T \Sigma_{w^\star}^+ \mat{X}^T (Y - X \hat{\beta}).
\end{align*}
After conditioning on $E$ we have
\begin{align*}
\p( \hat{\gamma} = c^T \Sigma_{w^\star}^+ \mat{X}^T \varepsilon / n \mid E) \rightarrow 1.
\end{align*}
This implies that with probability $\p(E) \approx 1/e$, the de-biased estimator $\hat{\gamma}$ is not consistent. Finally, note that this result does not depend on the specific choice of the relaxed inverse used in place of $\Sigma_{w^\star}^{-1}$.

\end{example}

The reason for the bias in this example is that the restricted eigenvalue of $\hat{\Sigma}$ is zero with non-negligible probability. 
A simple solution to this problem is to constrain the population covariance matrix to have eigenvalues that are bounded away from zero.
This motivates considering designs where for pre-specified $0<\lambda_\star \leq \lambda^\star<\infty$ the covariance matrix is constrained to have eigenvalues that are bounded away from $0$ and $\infty$:
\begin{align}\label{problem:2}
\mathbf{P2 : }\min_{w \in \R^{N}}  \quad &c^T \Sigma^{-1} c\\[-5pt]
\text{s.t.}  \quad 
&\Sigma = \sum_{1 \leq i \leq N} w_i x_i x_i^T\\[-5pt]
& \lambda_\star I \preccurlyeq \Sigma \preccurlyeq  \lambda^\star I  \\[-5pt]
&\sum_{1 \leq i \leq N} w_i = 1, \quad w \geq 0.
\end{align}

While the original problem was equivalent to an LP, we show (in Proposition \ref{prop:SDP}) that the optimization problem $P2$ can be recast as the following semidefinite program (SDP):
\begin{align*}
\mathbf{P2^\prime : }\min_{t\in \R, w \in \R^N} \quad & t\\[-10pt]
\text{s.t.} \quad 
&\Sigma = \sum_{i=1}^N w_i x_i x_i^T, \quad \sum_{i=1}^N w_i = 1\\[-3pt]
& \lambda_\star I \preccurlyeq \Sigma \preccurlyeq  \lambda^\star I, \quad  \begin{bmatrix}
t & c^T \\
c & \Sigma
\end{bmatrix} \succcurlyeq 0\\[-5pt]
&w \geq 0.
\end{align*}

\begin{remark}[The advantanges of randomization.]
Randomized designs have important computational and statistical advantages in the high-dimensional setting. To see this, let us consider a natural notion of $c$-optimality for deterministic experiments. If $n_i\in \{0, 1, \dots, n\} $ is the number of times $x_i$ is repeated in the sample, then the variance of the de-biased lasso estimator is approximately equal to 
\begin{align*}
\frac{1}{n} c^T \hat{\Omega} \hat{ \Sigma} \hat{\Omega}^T c,
\end{align*}
where $\hat{\Sigma} = n^{-1} \sum_1^N n_i x_i x_i^T$ and $\hat{\Omega}$ is a relaxed inverse of $\hat{\Sigma}$ that is typically required to satisfy $\| c^T \hat{\Omega} \hat{\Sigma} - c^T \|_\infty \lesssim \sqrt{\log(p)/n}$ to ensure that the bias the of de-biased lasso estimator remains $o(1/\sqrt{n})$. It is then natural to attempt to minimize the the above variance subject to the aforementioned constraint (see for example a similar formulation for $D$-optimality in the work of \cite{huang2020optimal}):
\begin{align*}
\min_{\stackrel{(n_i)_1^N \in [n]^N} {\hat{\Omega} \in \mathbf{R}^{p\times p}}} &c^T \hat{\Omega} \hat{ \Sigma} \hat{\Omega} c \\[-10pt]
s.t.\quad & \| c^T \hat{\Omega} \hat{\Sigma} - c^T \|_\infty \lesssim \sqrt{\frac{\log(p)}{n}}\\[-10pt]
& \hat{\Sigma} = \frac{1}{n} \sum_{i=1}^N n_i x_i x_i^T\\[-10pt]
& \sum_1^N n_i \leq n.
\end{align*}

Designs obtained using the above formulation suffer from two problems:
\begin{itemize}
\item \textbf{Computational feasibility. } The above is a non-convex problem in a high-dimensional space ($\mathbf{R}^{N + p^2}$) for which it can be very difficult to find global optima (even when we ignore the integer constraints on $n_i$). Randomization here allows to approximate $ c^T \hat{\Omega} \hat{ \Sigma} \hat{\Omega} c$ with $c^T \hat{\Omega} c$ for $\hat{\Omega} := \Sigma^{-1}$, which is a convex function of the design $(w_i)_1^N$.

\item \textbf{Statistical accuracy. } Even if a sufficiently good (locally) optimal point of the above problem is found, guaranteeing statistical accuracy for the resulting design may not be possible, as the restricted eigenvalue condition may not be satisfied for it. Note that verifying the RE condition is NP-hard \citep{dobriban2016regularity}, so it is not known how to constrain the design in the above problem to satisfy the RE condition. Randomization solves this problem by using linear matrix inequalities (LMI) on the population covariance matrix, casting the problem as a semidefinite problem. The RE condition then is guaranteed to hold with high probability for the sample design matrix.
\end{itemize}

\end{remark}

\section{Application}\label{sec:application}

\begin{figure}
\includegraphics[width=\linewidth]{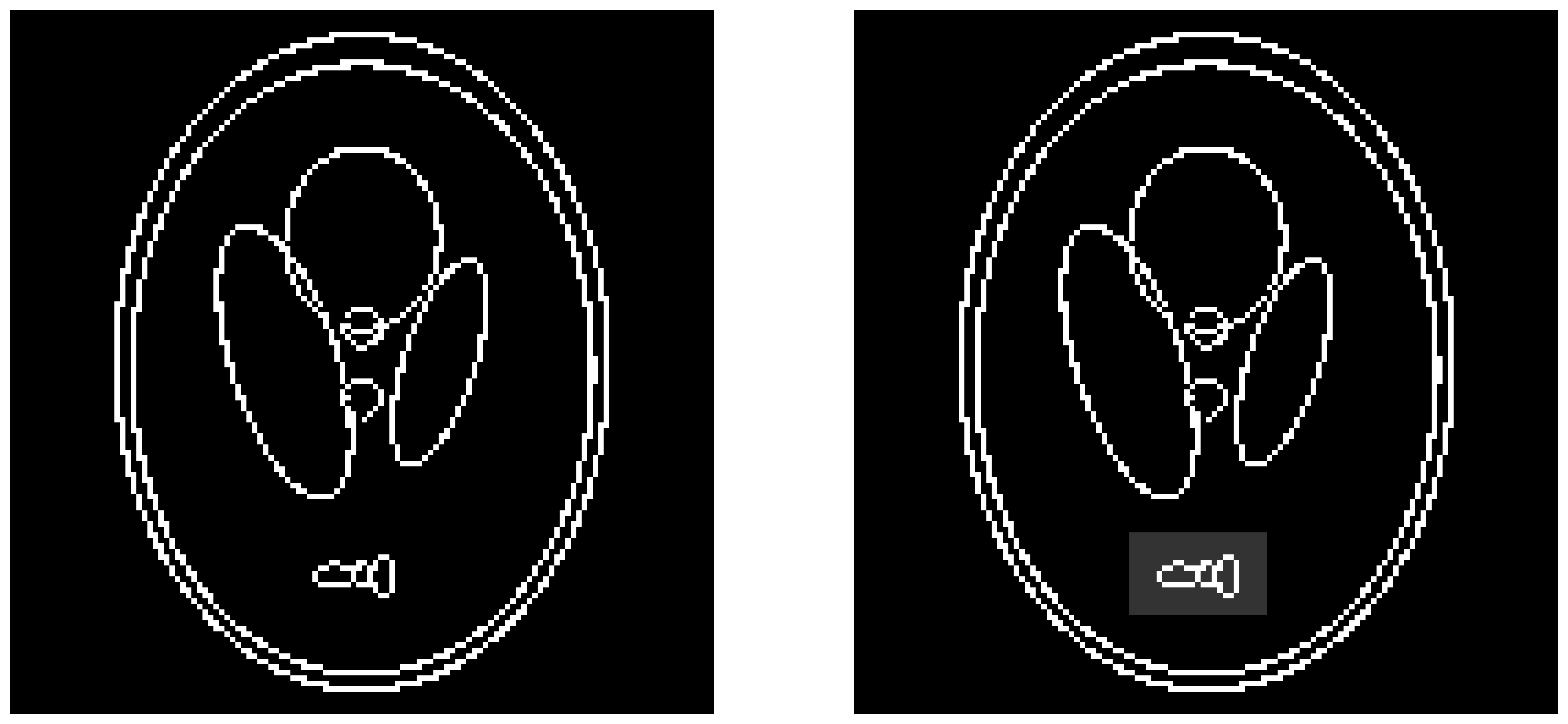}
\caption{The sparsified Shepp-Logan phantom as the underlying truth for our experiments (left).  The gray rectangle shows the region whose average intensity $(\gamma)$ is the target of inference.(right) }
\label{img:shepp_logan}
\end{figure}

Our example application is inspired by sparse magnetic resonance imaging (MRI) \citep{lustig2007sparse}. At a high level, the measurements in an MRI application form a noisy Fourier transform of an underlying quantity (typically proton density of water molecules in a cross-section of body):
\begin{align*}
y = \mathcal{F} \cdot \vect( \theta^\star)+ \varepsilon,
\end{align*}
where $\beta^\star := \vect(\theta^\star)$ is a vectorized (i.e. one-dimensional) representation of the two-dimensional underlying truth $\theta^\star$, and $\mathcal{F}$ is the matrix of the discrete Fourier transform (in the same basis that $\beta^\star$ is represented), and $\varepsilon$ is measurement noise taken to have a $N(0,1)$ distribution. For an $n \times p$ matrix $A$, the vectorized representation $\vect(A)$ is obtained by stacking the columns of $A$ in an $np$ vector as follows
\begin{align*}
\vect(A) = (A_{11}, \dots, A_{n1}, A_{12}, \dots, A_{n2}, \dots, A_{1p}, \dots, A_{np})^T.
\end{align*}
In practice, obtaining the MRI measurements is slow and time-consuming, a problem that has inspired a vast literature (under the umbrella of ``compressed sensing'') on methods allowing reduction of the number of such measurements while preserving the quality of the reconstructed image. If $\beta^\star$ is sparse, one of the most well-known methods of reconstruction is the Lasso
\begin{align*}
\hat{\beta} := \arg\min_{\beta'} \frac{1}{2n} \| y - \mathcal{F} \beta' \|_2^2 + \lambda \| \beta' \|_1.
\end{align*}
Often, $\beta^\star$ is sparse, not in the original basis, but rather in a different basis, e.g. the wavelet basis $W$. In this case one solves the transformed problem \citep{chen2001atomic}:
\begin{align*}
\hat{\beta} := \arg\min_{\beta'} \frac{1}{2n} \| y - \mathcal{F} \beta' \|_2^2 + \lambda \| W \beta' \|_1.
\end{align*}
Alternatively, if $\beta$ has a small total variation, for example if the image comprises large piece-wise constant parts, one can use a total variation penalty term \citep{rudin1992nonlinear}:
\begin{align*}
\hat{\beta} := \arg\min_{\beta'} \frac{1}{2n} \| y - \mathcal{F} \beta' \|_2^2 + \lambda \| \nabla \beta' \|_1.
\end{align*}
where $\nabla$ is a discrete gradient operator that yields sparse representations of the class of images under study.

\begin{figure}\label{img:hist_compare}
\centering
\includegraphics[width=0.9\linewidth]{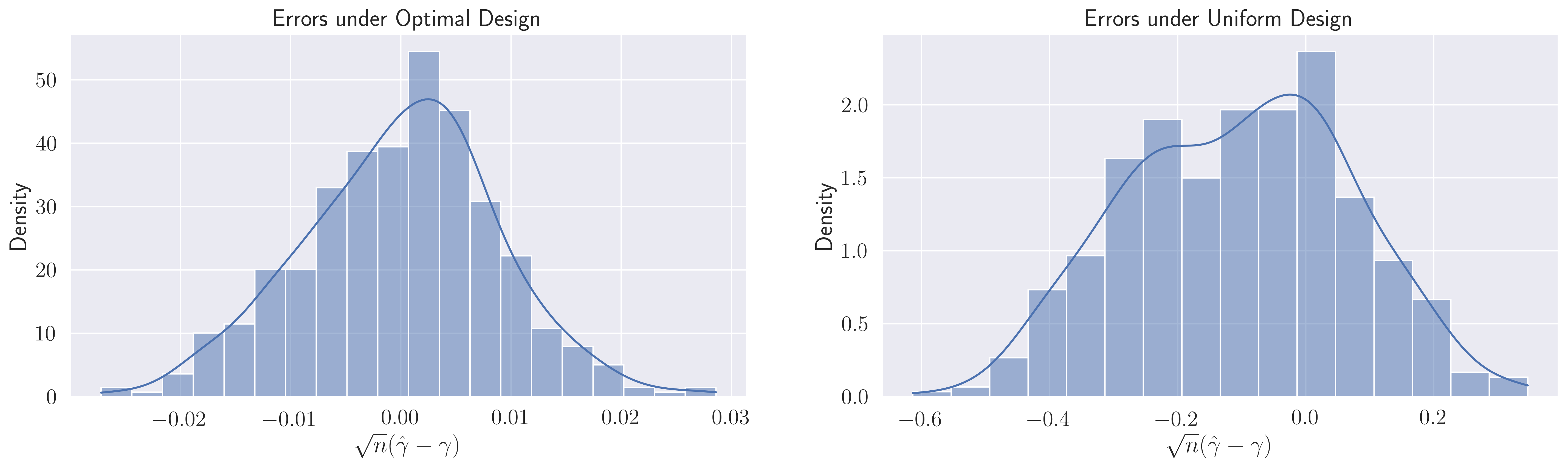}
\caption{Comparison of $\sqrt{n}$-scaled estimation errors $\sqrt{n}(\hat{\gamma} - \gamma)$ for the optimal and uniform designs }
\end{figure}

In order to make the discussion of our experiments more straightforward, we make the following simplifications:
\begin{itemize} 
\item We start with a sparse image (Figure \ref{img:shepp_logan}[left]), so that we do not need to take the sparsifying transforms (wavelet, gradient, etc.) into account. If we denote the original image (in our case, the Shepp-Logan phantom) by $\iota$, then the sparse image is taken to be the sign of the Laplacian of the original image, that is, $\theta^\star = \mathbf{1}(\Delta \iota > 0)$ where $\Delta \iota := \partial_x^2\iota + \partial_y^2\iota $ is the Laplacian operator  typically used for edge detection in image processing and where the $\partial_x, \partial_y$ are Sobel derivatives (see \cite[Chapter 10]{MISRA2020289} for details). In the resulting image only $6.3\%$ of the pixels are non-zero while in the original image $43\%$ are non-zero.
\item We use the discrete cosine transform (DCT) \citep{ahmed1974discrete} to avoid complex-valued Fourier measurements that are obtained in MRI applications. Alternatively, one could also consider the Fourier transform as pairs of cosine and sine transforms, resulting in a design matrix with twice as many rows as the (complex) Fourier transform matrix.
\end{itemize}

The number of pixels in the image on the left of Figure \ref{img:shepp_logan} is $128 \times 128$. Thus for this application $N = p = 128^2 = 16384$ and we set $n = N / 2$, so that we are under-sampling the DCT coefficients by $50\%$. The design matrix is the matrix of 2-dimensional discrete cosine transform of type 2, given by $X = \sqrt{N} \cdot (D \otimes D)^T$ where $\otimes$ denotes the Kronecker product\footnote{Here $D$ is the matrix of DCT (type 2) for 1-dimensional signals, and the Kronecker product computes the 2-dimensional DCT matrix for vectorized images.} and
\begin{align}\label{def:DCT}
D_{ij} = \begin{cases}
\frac{1}{\sqrt{m}}  & :\, i = 1\\
\sqrt{\frac{2}{m}} \cos \left( \frac{\pi (i - 1) (2j - 1)}{2 m} \right) & :\, 2 \leq i \leq m = \sqrt{p}.
\end{cases}
\end{align}
The scaling by $\sqrt{N}$ is done to ensure $X^T X / N = I_N$, meaning that under uniform sampling, the population design covariance is the identity matrix (and thus has well-behaved extreme eigenvalues).

The $c$ vector is taken to be $c = \vect(\tilde{c})$ where:
\begin{align*}
\tilde{c}_{ij} &= \begin{cases}
\frac{1}{375} &: \, 95 \leq i < 110 \text{ and } 50 \leq j < 75\\
0 &:\, \text{otherwise.}
\end{cases}
\end{align*}
This choice of $c$ results in $\gamma = \langle c, \beta^\star \rangle = \langle \tilde{c}, \theta^\star \rangle$ being the average intensity of the image on the support of $\tilde c$ (the gray rectangle in Figure \ref{img:shepp_logan}[right]), that is, $\gamma = \sum_{95 \leq i < 110} \sum_{50 \leq j < 75} \theta^\star_{ij} / 375$.
With the above choice of $c$, the parameter of interest $\gamma = \langle c, \beta^\star \rangle = 0.10667$. The optimal design is then obtained by solving the problem ${P1'}$.

Figure \ref{img:hist_compare} shows the histograms of estimation errors over 500 i.i.d. realizations of the estimates under uniform and optimal designs. The scaled bias and standard errors are presented in Table \ref{tab:bias_se}. As can be seen in the table, the optimal design leads to significant improvements in both bias and variance of the estimator. In practice, this translates to higher detection rates of weak signals in images and higher power for hypothesis tests.

\begin{table}
\centering
\begin{tabular}{ |c|c|c|c| } 
\hline
Design / Measure & $\sqrt{n} \times $ bias & $\sqrt{n} \times $ standard error \\
\hline
Optimal Design & $4.1564 \times 10^{-5}$ & 0.0088 \\ 
Uniform Design & -0.1042 & 0.1751 \\ 
\hline
\end{tabular}
\caption{Comparison of bias and standard error for the optimal and uniform designs}
\label{tab:bias_se}
\end{table}

\textbf{Simulated Data Example. }
In our second experiment, we compare the performance of the de-biased estimator under uniform and optimal design for an experimental domain generated from a multivariate normal distribution. For each combination of problem parameters, $N = 1000$ observations $(x_i)_1^N$ were drawn from $N(0_p, \Sigma_\kappa)$ where $(\Sigma_\kappa)_{ij} = \kappa^{|i - j|}$ and $\kappa \in \{ 0.1, 0.9 \}$ and $p = 500$. The regression parameter $\beta^\star$ with two sparsity levels $s \in \{ 5, 10 \}$ was taken to be
\begin{align*}
\beta^\star_j = \begin{cases}
\frac{s +1 - j}{s} &: j \leq s, \\
0 &: s < j \leq p.
\end{cases}
\end{align*}
Two different choice of $c \in \{ \mathrm{e}_s, \mathrm{e}_{s + 1} \}$ were considered that correspond to estimating $\gamma = \beta_s = 1/s$ and $\gamma = \beta_{s + 1} = 0$. The optimal design was found in each case using the linear program (\ref{linear_program}) (resulting in an optimal design with non-singular population covariance matrix). To approximate the distribution of the de-biased lasso estimator for each specification of the problem (i.e. choices of $\kappa$, $s$ and $c$), 500 Monte Carlo samples $\hat{\gamma}$ were generated as follows. Each time a sample $(X_i)_1^n$ of size $n = 200$ was drawn from $(x_i)_1^N$ according to the uniform or optimal design, the responses were generated according to $y_i = \langle \beta^\star, X_i \rangle + \varepsilon_i$ with $\varepsilon_i \stackrel{iid}{\sim} N(0,1)$, and finally the de-biased 
 $95\%$ confidence intervals were generated using
\begin{align*}
\hat{\gamma} \pm z_{0.025} \cdot \sqrt{\frac{ c^T \Sigma^{-1} \hat{\Sigma} \Sigma^{-1} c}{n}}
\end{align*}
with $z_{0.025}$ denoting the upper 2.5 percentile of $N(0,1)$ and where $\Sigma = N^{-1} \sum_1^N x_i x_i^T$ in the case of uniform design and $\Sigma = \sum_1^N w_i^\star x_i x_i^T$ under the optimal design $w^\star$. In the case $\gamma = \beta_s = s^{-1}$ (which we call the ``signal'' scenario), power was computed as the proportion of times (out of the 500 realizations) the $95\%$ confidence interval did not include zero. In the case $\gamma = \beta_{s + 1} = 0$ (which we call the ``noise'' scenario), the false positive rate is computed as the proportion of times (out of the 500 realizations) the $95\%$ confidence interval does not include zero. In both cases the coverage is the proportion of times the confidence interval covers the parameter of interest. The standard errors were computed as the sample standard deviation of the 500 realizations of $\hat{\gamma}$ in each case.

The results of the experiment are presented in Table \ref{gauss_tab_signal} and Table \ref{gauss_tab_noise} for the signal and noise scenarios, respectively. As can be seen in Table \ref{gauss_tab_signal}, the optimal design results in improved power while controlling the coverage close to the nominal level.

\begin{table}[t]
\centering
\begin{tabular}{|lcccc|} \hline
($\kappa$, s) & Design &  Power &   Coverage & $\sqrt{n} \times$Standard Error\\ \hline
\multirow{2}{*}{ (0.1, 5)   }  &  U&  0.496&   0.952 & 1.42\\
            &  O&  0.806&  0.928 & 1.00\\ \hline
\multirow{2}{*}{(0.1, 10)  } & U &  0.152&  0.936 & 1.49\\
                &O  & 0.29 &  0.926& 1.05\\ \hline
\multirow{2}{*}{(0.9, 5)   } & U &  0.096&  0.97 & 4.02\\
               &O  &0.156  & 0.938& 2.96 \\ \hline
\multirow{2}{*}{(0.9, 10) } &  U&  0.064&  0.938 & 4.46\\
                 &O  & 0.078 & 0.944 & 2.85\\ \hline
\end{tabular}
\caption{Performance of de-bisaed lasso under uniform and optimal designs for estimating $\gamma = \beta_s = s^{-1}$. $U$ and $O$ signify Uniform and Optimal design, respectively.  }\label{gauss_tab_signal}
\end{table}

\begin{table}[t]
\centering
\begin{tabular}{|lcccc|} \hline
($\kappa$, s) & Design &  False Positive Rate &   Coverage & $\sqrt{n} \times$Standard Error \\ \hline
\multirow{2}{*}{ (0.1, 5)   }  &  U&  0.054& 0.946 & 1.44\\ 
            &  O&  0.048& 0.952 & 0.97 \\ \hline
\multirow{2}{*}{(0.1, 10)  } & U &  0.058&  0.942 & 1.44\\
                &O  & 0.064 & 0.936  & 1.00\\ \hline
\multirow{2}{*}{(0.9, 5)   } & U &  0.048&  0.952 & 4.19\\ 
               &O  &0.042  & 0.958 & 2.84 \\ \hline
\multirow{2}{*}{(0.9, 10) } &  U&  0.046&  0.954 & 4.25\\ 
                 &O  & 0.056 & 0.944 & 2.91 \\ \hline
\end{tabular}
\caption{Performance of de-bisaed lasso under uniform and optimal designs for estimating $\gamma = \beta_{s+1} = 0$. $U$ and $O$ signify Uniform and Optimal design, respectively.  } \label{gauss_tab_noise}
\end{table}

\begin{remark}
For both of our experiments, we have used Problem $P1'$ to find the optimal design, which in both cases resulted in designs with non-singular population covariance matrices. In practice, for large-scale problems it is often much faster to solve the linear program $P1'$ versus the semidefinite program $P2$. Therefore one can first quickly solve $P1'$, and resort to $P2$ if the covariance matrix resulting from $P1'$ is (close to) singular.
\end{remark}

\section{Discussion and Future Work}
In this work we studied the problem of optimal design under the $c$-optimality criterion in the high-dimensional setting where $n < p$. We have proposed a semidefinite program that minimizes the variance of de-biased estimators while allowing the derivation of theoretical guarantees for the resulting design. We have shown that in contrast to the low-dimensional setting where the OLS estimator is unbiased, in the high-dimensional setting one needs to control the bias when minimizing the variance of de-biased estimators. The practical efficiency gains from these optimal designs are presented in simulation experiments and are quite impressive.

Our work focuses on the setting where $ p \leq N$. However the $N < p$ case can also arise in applications where the experimental domain is not as large and poses interesting challenges (such as the difficulty of enforcing RE conditions on the design as discussed in the introduction) for future work. Other directions for further research include the study of optimal designs for generalized linear models in a high-dimensional setting and under other optimality criteria than $c$-optimality. Yet another direction for extending our results is a study optimal designs for $A \beta$ for a matrix $A \in \R^{q \times p}$ with $q$ a fixed integer.

\section{Proofs}\label{sec:proofs}
For completeness, we record here Elfving's theorem \citep{elfving1952optimum} which has been alluded to in the main text.

\begin{theorem}[Elfving (1952)]
Let $(x_i)_1^N \in \mathbf{R}^p$ be the available design points and let $c \in \mathbf{R}^p$. Define the Elfving set to be the convex hull of $(\pm x_i)_1^N$:
\begin{align*}
\mathcal{E} := \mathrm{conv}\left( \{ x_i : 1 \leq i \leq N \} \cup \{ -x_i : 1 \leq i \leq N \}  \right).
\end{align*}
Let $x_c$ be the point on the boundary of $\mathcal{E}$ that intersects the half-line  passing through the origin and $c$:
\begin{align*}
x_c = \partial \mathcal{E} \cap \{ tc : t \geq 0\}.
\end{align*}
If we write $x_c = \sum_{i=1}^N v_i x_i$, then the $c$-optimal design is given by $w_i^\star := |v_i| / \sum_1^N |v_j|$.

\end{theorem}

\begin{definition}[Restricted Eigenvalue Condition]\label{RE-def}
A matrix $A$ is said to satisfy the restricted eigenvalue condition  $RE(s_0, k_0, A)$with parameter $\lambda_{RE}$ if
\begin{align*}
\lambda_{RE} := \quad \min_{\stackrel{J \subset \{ 1, \dots, P\}}{|J| \leq s_0}}  \min_{ \stackrel{ \|v_{J^c} \|_1 \leq k_0 \| v _J \|_1}{v \neq 0} } \frac{\|Av \|_2}{\|v_J\|_2} > 0.
\end{align*}
We also denote this quantity by $\lambda_{RE}(s_0, k_0, A)$.
\end{definition}

We record here a theorem by \cite{rudelson2012reconstruction} that relates the restricted eigenvalues of random matrices to the (restricted) eigenvalues of the corresponding population covariance matrices. In the theorem the smallest $k$-sparse eigenvalue of $A$ is defined as 
\begin{align*}
\rho_{\min}(k, A) = \min_{\stackrel{\| t \|_0 \leq k}{t \neq 0}} \frac{ \| At \|_2}{\| t\|_2}.
\end{align*}

\begin{theorem}[Rudelson \& Zhou, Theorem 8]\label{theorem:RZ}
Let $0 < \delta < 1$ and $ 0 < s_0 < p$. Let $X \in \mathbf{R}^p$ be a random vector such that $\| X \|_\infty \leq M$ a.s. and denote $\Sigma = \E X X^T$. Let $\mathbb{X}$ be an $n \times p$ matrix whose rows $X_1, X_2, \dots, X_n$ are independent copies of $X$. Let $\Sigma$ satisfy the $RE(s_0, 3k_0,  \Sigma^{\frac{1}{2}})$ condition as in Definition \ref{RE-def}. Define
\begin{align*}
d = s_0 \left( 1 + \max_j \|\Sigma^{\frac{1}{2}} e_j\|_2^2 \frac{ 16 (3k_0)^2 (3k_0+1) }{\delta^2 \cdot \lambda^2_{RE}(s_0, k_0, \Sigma^{\frac{1}{2}})}   \right).
\end{align*}
Assume that $d \leq p$ and $\rho = \rho_{\min}(d, \Sigma^{\frac{1}{2}}) > 0$. Assume that the sample size $n$ satisfies
\begin{align*}
n \geq n_0 := \frac{C_{RZ}M^2 d \cdot \log p }{\rho^2 \delta^2} \cdot \log^3\left( \frac{ C_{RZ}M^2 d \cdot \log p}{ \rho^2 \delta^2} \right),
\end{align*} 
for an absolute constant $C_{RZ}$.
Then with probability at least $1 - \exp\left( -\delta \rho^2 n / (6M^2d) \right)$, the $RE(s_0, k_0, \mathbb{X}/\sqrt{n})$ condition holds for matrix $\mat{X} / \sqrt{n}$ with $ \lambda_{RE}(s_0, k_0, \mathbb{X}/\sqrt{n}) \geq   (1 - \delta) \cdot \lambda_{RE}(s_0, k_0, \Sigma^{\frac{1}{2}})$.

\end{theorem}

Note that this theorem concerns observations obtained via i.i.d sampling. The next proposition shows that a similar result holds for Poisson sampling as used in our work. Since the usual Lasso guarantees require $RE(s_0, k_0 = 3, \mathbb{X}  / \sqrt{n})$, we will be using the above theorem with $k_0 = 3$.

\begin{proposition}\label{rudelson_poisson}
Suppose that $K_1, \dots, K_N$ are independent Poisson random variables with $K_j \sim \operatorname{Poisson}(n w_j)$ and $\sum_1^N w_j = 1$, so that $K := \sum_1^N K_j \sim \operatorname{Poisson}(n)$. Let $\mathbb{X}$ be a $K \times p$ matrix where $x_j^T$ is repeated in the rows of $\mathbb{X}$ exactly $K_j$ times. Suppose that
 \begin{itemize}

\item The population covariance matrix $\Sigma = \sum_{j=1}^N w_j x_j x_j^T$ satisfies
\begin{align*}
\lambda_\star \leq \lambda_{\min}(\Sigma) \leq \lambda_{\max}(\Sigma) \leq \lambda^\star.
\end{align*}

  \item The expected sample size $n$ satisfies
\begin{align*}
n \geq \frac{5}{4} \tilde{n}_0 \quad &\text{ where } \tilde{n}_0 =  \frac{\tilde{C} M^2  \lambda^\star s_0 \log p}{\lambda_\star^2}\log^3 \left( \frac{\tilde{C} M^2  \lambda^\star s_0 \log p}{\lambda_\star^2} \right),\\
&\text{ and } \tilde{C} = 4 \times 51841 C_{RZ}.
\end{align*}

\end{itemize}
Then with probability at least $1 - e^{-\frac{\tilde{n}_0}{4}} - e^{-\frac{n_0 \lambda_\star}{12M^2 d}}$ we have $\lambda_{RE}(s_0, 3, \mathbb{X}/\sqrt{K}) \geq \lambda_{RE}(s_0, 3, \Sigma^{\frac{1}{2}}) / 2$.
\end{proposition}
\begin{proof}
The basic idea is that conditioned on the total number of samples $K$, the conditional distribution of $(K_1, \dots, K_N)$ is multinomial with success probabilities $(w_1, \dots, w_N)$. Thus conditionally, the rows of $\mathbb{X}$ form an i.i.d. sample of the population $(x_i)_{i=1}^N$ with probabilities $(w_i)_{i=1}^N$. Therefore, the result of Theorem \ref{theorem:RZ} can be applied to get a lower bound on the restricted eigenvalue of $\mathbb{X} / \sqrt{K}$. For this, we first find upper bounds on $d,n_0$ and a lower bound on $\rho, \lambda_{RE}$ as needed in the theorem.

From the assumption on the spectrum of $\Sigma$ and the definitions of sparse and restricted eigenvalues it is clear that $\rho^2 \geq \lambda_\star$ and $\lambda_{RE}^2(s_0, 3, \Sigma^{\frac{1}{2}}) \geq \lambda_\star$. From these inequalities, and using $\delta = 1/2$ and $k_0 = 3$, we obtain an upper bound on $d$:
\begin{align*}
d &\leq s_0 \left( 1 + \frac{\lambda^\star}{\lambda_\star} 4 \cdot 64 \cdot 9^2 \cdot 10 \right)\\
& \leq 51841 \cdot s_0 \frac{\lambda^\star}{\lambda_\star} .
\end{align*}

Next, writing the $n_0$ in Theorem \ref{theorem:RZ} as $m_0  \log^3(m_0)$, we can bound $m_0$ by
\begin{align*}
m_0 &= \frac{C_{RZ}M^2 d \cdot \log p }{\rho^2 \delta^2}\\
&\leq \frac{4 \times 51841 C_{RZ} M^2  \lambda^\star s_0 \log p}{\lambda_\star^2} \\
& = \frac{\tilde{C} M^2  \lambda^\star s_0 \log p}{\lambda_\star^2}
\end{align*}
It follows that 
\begin{align*}
\tilde{n}_0 \geq  m_0 \log^3 m_0 = n_0.
\end{align*}

Next, we show that with high probability, the sample size $K$ is not smaller than $n_0$. We have
\begin{align*}
\p( K < n_0 ) \leq \p( K < \tilde{n}_0) = e^{-n} \sum_{j=0}^{\tilde{n}_0-1} \frac{n^j}{j!} \leq e^{\tilde{n}_0 - n} \leq e^{-\frac{\tilde{n}_0}{4}}.
\end{align*}

Now we proceed by conditioning on $K = k$ for $k \geq n_0$. Note that as mentioned before, given $K = k$, the rows of $\mathbb{X}$ have the same distribution as a weighted i.i.d. sample from $(x_i)_1^N$ with probabilities $(w_i)_1^N$, and since $ k \geq n_0$, by Theorem \ref{theorem:RZ} the probability that the $RE(s_0, 3, \mathbb{X}/\sqrt{k})$ does not hold is at most $\exp(-k \lambda_\star / (12 M^2 d))$. Denote by $B$ the event that $\lambda_{RE}(s_0, 3, \mathbb{X}/\sqrt{K}) < \lambda_{RE}(s_0, 3, \Sigma^{\frac{1}{2}}) / 2$. Then we have
\begin{align*}
\p\left( B \right) &\leq \p( K < n_0 ) + \p(B \cap [K \geq n_0]) \\
&\leq e^{-\frac{\tilde{n}_0}{4}} + \sum_{k=n_0}^\infty \p( B \mid K = k) \cdot \p( K = k) \\
&\leq e^{-\frac{\tilde{n}_0}{4}} + \sum_{k=n_0}^\infty \exp\left(-\frac{k \lambda_\star}{12M^2 d}\right) \cdot \p( K = k)\\
&\leq e^{-\frac{\tilde{n}_0}{4}} + e^{-\frac{n_0 \lambda_\star}{12M^2 d}}.
\end{align*}

\end{proof}

Next we record here a conditional central limit theorem due to \cite[Theorem 1 and Corollary 3]{Bulinski2017ConditionalCL} that will be used in the proof of theorem 1. Let  $\{U_{n,k}\}_{n,k}$ be a triangular array and $\A_n$ be the $\sigma$-algebra that can change with $n$. Denote by $\E^{\A_n}[\cdot] = \E[ \cdot \mid {\A_n}]$ the conditional expectation with respect to $\A_n$ and define
\begin{align*}
S_n := \sum_{i=1}^N U_{n,k} , \quad \quad (\sigma_n^{\A_n})^2 :=  \V^{\mathcal{A}_n} S_n =\E^{\A_n} (S_n - \E^{\A_n} S_n)^2
\end{align*}

\begin{theorem}[Theorem 1 and Corollary 3 of Bulinski, 2017]\label{bulinski}
Let $\{ U_{n,k} : k= 1, \dots, k_n, n \in \mathbf{N} \}$ be an array of random variables, which are $\mathcal{A}_n$-independent (i.e. independent given $\A_n$) in each row (for some $\sigma$-algbera $\mathcal{A}_n \subset \mathcal{F}$, where $n \in \mathbf{N}$), and $\V^{\mathcal{A}_n}( U_{n,k} ) < \infty$ (a.s.) for $ k = 1, \dots, k_n, n \in \mathbf{N}$. Assume that $(\sigma_n^{\mathcal{A}_n})^2 := \V^{\mathcal{A}_n} S_n> 0 $ (a.s.) for all n large enough. Then the two relations
\begin{align}
\max_{k=1, \dots, k_n} \frac{ \V^{\mathcal{A}_n }U_{n,k}}{ (\sigma_n^{\A_n})^2} \rightarrow_p 0
\end{align}
and
\begin{align*}
\E^{\A_n} \exp \left\{ it \frac{ S_n - \E^{\A_n} S_n }{ \sigma_n^{\A_n} } \right\} \rightarrow_p \exp\left\{ - \frac{t^2}{2}\right\} , \quad n \rightarrow \infty.
\end{align*}
hold if and only if the $\A_n$-Lindeberg condition is satisfied in a weak form: for any $t > 0$
\begin{align*}
 \quad T_n := \frac{1}{(\sigma_n^{\A})^2} \sum_{i = 1}^N \E^A \left[ (U_k - \E^\A U_k)^2 \mathbf{1} \{ |U_k - \E^\A U_k| > t \sigma_n^{\A}  \}  \right] \rightarrow_p 0.
\end{align*}
Furthermore, if the above $\A_n$-Lindeberg condition holds, then we have
\begin{align*}
\frac{S_n - \E^{\A_n} S_n}{ \sigma_n^{\A_n} } \rightarrow_d Z \sim N(0,1), \quad \text{ as } n \rightarrow \infty.
\end{align*}

\end{theorem}

\textbf{Proof of Theorem 1.} We present the proof in three parts:

\textbf{Part 1. (Bias Bound) }
First note that using weighted lasso with weights $\widehat{W}_j = \sum_i X_{ij}^2 / n$ is equivalent to normalizing the columns of $\mat{X}$ before applying the lasso. Furthermore, since $\E \widehat{W}_j^2 = \Sigma_{jj}$ and $0 < \lambda_\star \leq \Sigma_{jj} \leq \lambda^\star < \infty$ for all $j \in [p]$, the uniform boundedness of $X_{ij}$ and a standard concentration argument yields 
\begin{align*}
\max_{1 \leq j\leq p} \left| \frac{\widehat{W}_j^2}{\Sigma_{jj}} - 1 \right| \rightarrow_p 0
\end{align*}
Therefore with high probability the weights $\widehat{W}_j$ are bounded away from $0, \infty$ and thus the standard (unweighted) lasso guarantees apply (for details see \cite[Section 6.9]{buhlmann2011statistics}).
Next, observe that since $0 < \lambda_\star \leq \lambda_{\min}(\Sigma) \leq \lambda_{\max}(\Sigma) \leq \lambda^\star < \infty$ and because by assumption $\sqrt n \gg s \log^{3/2}(p)$, Proposition \ref{rudelson_poisson} guarantees that the scaled design matrix $\mat{X}/\sqrt{K}$ satisfies the RE condition with a restricted eigenvalue larger than $\sqrt{\lambda_\star}/2$ with probability $1 - o(1)$, where $K$ is the number of samples obtained via Poisson sampling. Calling this event $G$, we have $P(G) = 1 - o(1)$. Given $G$, we have\footnote{note that the lower bound $\sqrt{\lambda_\star}/2$ on the restricted eigenvalue is uniform over $G$}
\begin{align*}
\p\left( \| \hat{\beta} - \beta\|_1 \lesssim \frac{\sigma_\varepsilon s}{\lambda_*} \sqrt{\frac{\log(p)}{K}}  \,\middle\vert\, G \right) = 1 - o(1).
\end{align*}
 (for a proof see for example Theorem 6.1 and Corollary 6.6 of \cite{buhlmann2011statistics} for proofs).
 It follows that 
 \begin{align*}
 \p\left( \| \hat{\beta} - \beta\|_1 \lesssim \frac{\sigma_\varepsilon s}{\lambda_*} \sqrt{\frac{\log(p)}{K}} \right) &\geq \p\left( \| \hat{\beta} - \beta\|_1 \lesssim \frac{\sigma_\varepsilon s}{\lambda_*} \sqrt{\frac{\log(p)}{K}}  \,\middle\vert\, G \right) \cdot \p(G) \\
 &= (1 - o(1)) \cdot (1 - o(1))\\
 &= (1 - o(1)).
 \end{align*}
 Next, note that since $K / n \rightarrow_p 1$, we can substitute $n$ for $K$ in the above bound and write
 \begin{align}\label{lasso_bound}
  \p\left( \| \hat{\beta} - \beta^\star\|_1 \lesssim \frac{\sigma_\varepsilon s}{\lambda_*} \sqrt{\frac{\log(p)}{n}} \right) \rightarrow 1.
 \end{align}

Let $\hat{w}_i = N_i / n$ where $N_i$ are obtained using Poisson sampling. Then the debiased lasso estimator (defined in \ref{deb_lasso}) can be written as
\begin{align*}
    \hat{\gamma} &= \langle c, \hat{\beta} \rangle + u^T\hat{\Sigma} (\beta^\star - \hat{\beta}) + \frac{1}{n} u^T X^T \varepsilon\\
    &= \gamma + c^T(\Sigma^{-1}\hat{\Sigma} - I)(\beta^\star - \hat{\beta}) +\frac{1}{n} c^T\Sigma^{-1} X^T \varepsilon.
\end{align*}
Subtracting $\gamma$ from both sides and multiplying by $\sqrt{n}$ we obtain
\begin{align*}
    \sqrt{n}(\hat{\gamma} - \gamma) = \sqrt{n} c^T (\hat{\Sigma}^{-1}\hat{\Sigma} - I)(\beta^\star - \hat{\beta}) + \frac{1}{\sqrt{n}} c^T\Sigma^{-1} X^T \varepsilon.
\end{align*}
We show that the first term is $o_p(1)$.  Using an $\ell_1 - \ell_{\infty}$ bound and the error rate of the lasso estimate, with probability $1 - o(1)$ we have
\begin{align*}
    \sqrt{n} |c^T (\hat{\Sigma}^{-1}\hat{\Sigma} - I)(\beta^\star - \hat{\beta})|
    &\leq  \sqrt{n} \| c^T (\hat{\Sigma}^{-1}\hat{\Sigma} - I)\|_\infty \cdot \| \hat{\beta} - \beta^\star \|_1 
\end{align*}
Thus we need first to upper bound $ \| c^T (\hat{\Sigma}^{-1}\hat{\Sigma} - I)\|_\infty$. Let $\hat{w}_i = N_i / n$ and note that
\begin{align*}
    c^T (\Sigma^{-1} \hat{\Sigma} - I) = \sum_{i=1}^N (\hat{w}_i - w_i^\star) c^T\Sigma^{-1} x_i x_i^T.
\end{align*}
Recall that $N_i$ is a Poisson random variable with mean $n w_i^\star$, and therefore $N_i - n w_i^\star$ is subexponential with 
\begin{align*}
    \Ex \exp( t (N_i - n w_i^\star)) &= \exp( n w_i^\star (e^t - 1 - t))\\
    &\leq  \exp(n w_i^\star( e^t t^2 / 2 )\\
    &\leq \exp(n w_i^\star( t^2)) \quad \text{ for } |t| \leq \frac{1}{2}.
\end{align*}
This shows that $\| N_i - nw_i^\star \|_{\psi_1} \lesssim \sqrt{n w_i^\star}$ \citep[Proposition 2.7.1]{vershynin2018high} , and therefore, $\| \hat{w}_i - w_i^\star\|_{\psi_1} \lesssim \sqrt{ w_i^\star / n}$. Define $V_{ij} = (\hat{w}_i - w_i^\star) c^T \Sigma^{-1} x_i x_{ij}$. Using Bernstein's inequality for subexponential random variables \citep[Theorem 2.8.1]{vershynin2018high}, for some absolute constant $b>0$ and all $j = 1, \dots, p$ we have
\begin{align}\label{bernstein}
    \p \left( \left| \sum_{i=1}^N V_{ij} \right| > t \right) &\leq 2 \exp\left( -b \min \left\{  \frac{t^2}{\sum_{i=1}^N \|V_{ij}\|_{\psi_1}^2 }, \frac{t}{\max_i \|V_{ij}\|_{\psi_1}}\right\} \right).
\end{align}
Using the bound $\max_{i,j}|x_{ij}| \leq M$, we have 
\begin{align*}
\sum_i \| V_{ij} \|_{\psi_1}^2 &\leq \frac{M^2}{n} \sum_i w_i^\star (c^T \Sigma^{-1} x_i)^2\\
&= \frac{M^2}{n} \sum_i w_i^\star c^T \Sigma^{-1} x_i x_i^T \Sigma^{-1} c\\
&=  \frac{M^2}{n} c^T \Sigma^{-1}(\sum_i w_i^\star  x_i x_i^T )\Sigma^{-1} c\\
&= \frac{M^2}{n} c^T \Sigma^{-1} c.
\end{align*}
Similarly, 
\begin{align*}
\max_i \|V_{ij}\|_{\psi_1} \leq M \cdot \max_i \sqrt{w_i^\star / n} |c^T \Sigma^{-1} x_i| \leq M \sqrt{\frac{c^T \Sigma^{-1} c }{n}}. 
\end{align*}
Using these bounds and for $t = \sqrt{2\log(p) / (nb)}$ the Bernstein bound (\ref{bernstein}) implies
\begin{align*}
    \p \left( \left| \sum_{i=1}^N V_{ij} \right| > \frac{2M\log(p)}{b}\sqrt{\frac{ c^T \Sigma^{-1} c}{n}} \right) &\leq 2 \exp\left( - \min \left\{\frac{4\log^2(p)}{b}, 2\log(p)\right\} \right).
\end{align*}
For $p > \exp(b/2)$, the exponential tail prevailes, and we obtain
\begin{align*}
    \p \left( \left| \sum_{i=1}^N V_{ij} \right| > \frac{2M\log(p)}{b}\sqrt{\frac{ c^T \Sigma^{-1} c}{n}} \right) \leq 2 p^{-2}, \text{ for all } j = 1, \dots, p.
\end{align*}
A union bound over all $j=1, \dots, p$ now yields
\begin{align}\label{l_inf_bound}
    \p \left(\max_{1\leq j \leq p} \left| \sum_{i=1}^N V_{ij} \right| > \frac{2M\log(p)}{b}\sqrt{\frac{ c^T \Sigma^{-1} c}{n}} \right) \leq 2 p^{-1} \quad, \text{ for } p > e^{b/2}.
\end{align}
Continuing with the $\ell_1 - \ell_{\infty}$ bound and using the upper bound (\ref{l_inf_bound}) and the error rate of the lasso estimate (\ref{lasso_bound}), with probability $1 - o(1)$ we have
\begin{align*}
    |\sqrt{n} |c^T (\hat{\Sigma}^{-1}\hat{\Sigma} - I)(\beta^\star - \hat{\beta})|
    &\leq  \sqrt{n} \| c^T (\hat{\Sigma}^{-1}\hat{\Sigma} - I)\|_\infty \cdot \| \hat{\beta} - \beta^\star \|_1 \\
    &\lesssim \sqrt{n} \cdot \left(  M \log(p) \sqrt{\frac{ c^T \Sigma^{-1} c}{n}} \right) \cdot \frac{ \sigma_\varepsilon}{\lambda_\star} s \sqrt{\frac{\log(p)}{n}}\\
    &= \frac{M \sigma_\varepsilon \sqrt{c^T \Sigma^{-1} c}}{\lambda_\star} \cdot \frac{s \log^{\frac{3}{2}}(p)}{\sqrt{n}}.
\end{align*}

\textbf{Part 2.(Variance Approximation) }Next, we prove the third part of the theorem as the argument used here will be useful in the proof of asymptotic normality. The conditional variance of the noise term is
\begin{align*}
    \V(\frac{1}{\sqrt{n}} c^T\Sigma^{-1}X^T \varepsilon \mid X) &= c^T \Sigma^{-1} \hat{\Sigma} \Sigma^{-1} c.
\end{align*}
We show that this variance can be approximated by $c^T \Sigma^{-1} c$, i.e.
\begin{align*}
    \frac{c^T \Sigma^{-1} \hat{\Sigma} \Sigma^{-1} c}{c^T\Sigma^{-1}c} \rightarrow_p 1.
\end{align*}

We assume $N = |\{i : w_i \neq 0\}|$ as otherwise one can throw out the zero weights. We want to show:
\begin{align*}
    A := \frac{\sum_i w_i (c^T\Sigma^{-1}x_i)^4}{n (c^T \Sigma^{-1} c)^2} \leq \frac{1}{n}.
\end{align*}
Let $d_i = \Sigma^{-1/2} x_i$ and $v = \Sigma^{-1/2} c$. Then $A$ can be written as
\begin{align*}
    A = \frac{ \sum_i w_i (d_i^T v)^4}{n \| v \|_2^4}.
\end{align*}

\textbf{Step 1.} First suppose that $N = p$. We have
\begin{align*}
    \sum_i w_i d_i d_i^T = \Sigma^{-\frac 1 2} \left(\sum_i w_i x_i x_i^T \right) \Sigma^{- \frac 1 2} = \mathbf{I}_p.
\end{align*}

(This is true for $N > p$ too.) Let $d_j^\star$ be the projection of $d_j$ on the ortho-complement of the span of $\{ d_i \mid i \neq j \}$. Then multiplying both sides by $d_j^\star$ we obtain
\begin{align*}
    w_j (d_j^T d_j^\star ) d_j = d_j^\star.
\end{align*}
Note that $d_j^T d_j^\star \neq 0$ for all $j = 1, \dots, p$ as $d_1, \dots, d_p$ form a basis for $R^p$ (since $\Sigma$ is non-singular by construction.) From this equation it follows that $d_1, \dots, d_p$ are orthogonal, and after multiplying both sides by $d_j$ one also obtains $\|d_j\|_2^2 = w_j^{-1}$.

Now since $A$ does not depend on $\|v\|_2$, we have
\begin{align*}
    A \leq \max_{u \neq 0} \frac{ \sum_i w_i (d_i^T u)^4}{n \| u \|_2^4} = \frac{1}{n} \max_{\|u\|_2^2 = 1} \sum_i w_i (d_i^T u)^4.
\end{align*}
The Lagrangian for the last maximization problem is
\begin{align*}
    L(u, \lambda) = \sum_i w_i (d_i^T u)^4 - 2 \lambda ( u^T u - 1).
\end{align*}
Taking derivative w.r.t. $u$ and setting to zero yields
\begin{align}\label{lagrange}
    \sum_i w_i (d_i^T \hat{u})^3 d_i = \lambda \hat{u}.
\end{align}
Changing variables to $\tilde{u} = \sqrt{1/ \lambda} \hat{u}$, we can rewrite (\ref{lagrange}) as
\begin{align}
    \sum_i w_i (d_i^T \tilde{u})^3 d_i = \tilde{u}.
\end{align}
Multiplying on the left once by $\tilde{u}^T$ and once by $d_j^T$ and using $d_j^T d_j = w_j^{-1}$ gives 
\begin{align}\label{conseq}
    \sum_i w_i (d_i^T \tilde{u})^4 = \tilde{u}^T \tilde{u} \quad \text{ and } \quad (d_j^T \tilde{u} )^2 = 1.
\end{align}
Note that $\sum_i w_i (d_i^T u)^4 / \|u\|_2^4$ does not depend on the norm of $u$, so that any nonzero multiple of $\hat{u}$, and in particular $\tilde{u}$, is a maximizer. Plugging $\tilde{u}$ in this expression and using (\ref{conseq}) we obtain
\begin{align*}
    A  &\leq \frac{ \sum_i w_i (d_i^T \tilde{u})^4}{n (\tilde{u}^T \tilde{u})^2}\\
    &\leq \frac{1}{n \cdot \sum_i w_i (d_i^T \tilde{u})^4} = \frac{1}{n}.
\end{align*}
This finishes the proof for the $N = p$ case.

\textbf{Step 2.} Now consider the $N > p$ case. The idea is to reduce this case to the $N = p$ case by appropriately extending the length of $d_i,v \in R^p$ from $p$ to $N$.

Note that the identity $\sum_i w_i d_i d_i^T = \mathbf{I}_p$ is still valid. Define the vectors $\tilde{d}_i \in R^N$ by $\tilde{d}_i^T = (d_i^T, f_i^T)$ for some vectors $f_i \in R^{N - p}$ such that
\begin{align*}
    \sum_i w_i \tilde{d}_i \tilde{d}_i^T = \mathbf{I}_N.
\end{align*}
A construction of these $f_i$'s is given in the Appendix.

For any $u \in R^N $ use a similar decomposition $u^T = (u_1^T, u_2^T)$ with $u_1 \in R^p$ and $u_2 \in R^{N - p}$. Then
\begin{align*}
n \cdot A &\leq \max_{\|v\|_2^2 = 1} \sum_i w_i (d_i^T v)^4 \\
&= \max_{\substack{ \|u\|_2^2 = 1 \\ u_2 = 0 }} \sum_i w_i (\tilde{d}_i^T u)^4\\
& \leq \max_{\|u\|_2^2 = 1} \sum_i w_i (\tilde{d}_i^T u)^4\\
& \leq 1,
\end{align*}
where the last inequality follows from the argument in \textbf{Step 1} (since now $\tilde{p} := \operatorname{dim}(\tilde{d}_i) = N$). This finishes the proof of the $N > p$ case.

\textbf{Part 3. (Asymptotic Normality)} Before we establish asymptotic normality, let us introduce some notation. Let $\A$ be the $\sigma$-algebra generated by $(N_i)_{i=1}^N$ (note that we are suppressing the dependence of $\A = \A_n$ on $n$ to lighten notation). As in Theorem \ref{bulinski}, denote by $\E^\A[\cdot] = \E[ \cdot \mid \A]$ the conditional expectation with respect to $\A$ and define 
\begin{align*}
U_{n,k} &:=\begin{cases} \frac{1}{\sqrt{n}} c^T \Sigma^{-1} x_k ( \varepsilon^{(k)}_1 + \dots + \varepsilon^{(k)}_{N_k}), \quad  & N_k > 0 \\
0, &  N_k = 0.
\end{cases}\\
S_n &:= \sum_{i=1}^N U_{n,k}, \quad
(\sigma_n^\A)^2 := \E^\A (S_n - \E^\A S_n)^2,
\end{align*}
where $\varepsilon_i^{(j)} \sim \varepsilon_n$ are iid mean-zero random variables with variance equal to $\E \varepsilon_n^2 = \sigma^2$ and sub-Guassian norm $\| \varepsilon_i^{(j)} \|_{\psi_2} \leq \sigma_\varepsilon$. It follows that in general $\sigma \lesssim \sigma_\varepsilon$.
Observe that since $\E^\A S_n = 0$, we have
\begin{align*}
(\sigma_n^\A)^2 = \E^\A S_n^2 &= \frac{1}{n} \sum_{k=1}^N (c^T \Sigma^{-1} x_k)^2 \E[( \varepsilon^{(k)}_1 + \dots + \varepsilon^{(k)}_{N_k})^2 \mid N_k]\\
&= \sum_{k = 1}^N \left( \frac{N_k}{n} \right) (c^T \Sigma^{-1} x_i)^2 \sigma^2 \\
&= \sigma^2 \cdot  c^T \Sigma^{-1} \hat{\Sigma} \Sigma^{-1} c.
\end{align*}
Thus we want to show

\begin{align*}
\frac{c^T \Sigma^{-1} \mat X^T \varepsilon}{\sqrt{n \sigma^2 \cdot  c^T \Sigma^{-1} \hat{\Sigma} \Sigma^{-1} c }  } =\frac{S_n - \E^{\A_n} S_n}{ \sigma_n^{\A_n} } \rightarrow_d Z \sim N(0,1), \quad \text{ as } n \rightarrow \infty.
\end{align*}
In light of Theorem \ref{bulinski}, it suffices to prove the following conditional Lindeberg condition is satisfied:
\begin{align*}
\forall t > 0: \quad T_n := \frac{1}{(\sigma_n^{\A})^2} \sum_{i = 1}^N \E^A \left[ (U_k - \E^\A U_k)^2 \mathbf{1} \{ |U_k - \E^\A U_k| > t \sigma_n^{\A}  \}  \right] \rightarrow_p 0.
\end{align*}
Next we find appropriate upper bounds on the summands in the Lindeberg condition.  Also note that by assumption the noise variance $\sigma^2$ is bounded away from $0$ and $\infty$, we can assume without loss of generality that $\sigma^2 = 1$ in what follows.
 Using the Cauchy-Schawrz inequality, 
\begin{align*}
\E^\A [ U_{n,k}^2 \mathbf{1} \{ |U_{n,k}| > t \sigma_n^{\A}  \} ] &\leq \left( (\E U_{n,k}^4)   ( \E^\A \mathbf{1} \{ |U_{n,k}| > t \sigma_n^{\A}  \}  )  \right)^\frac{1}{2} \\
&\leq \left( (\E^\A U_{n,k}^4 ) \cdot \frac{ \E^\A U_{n,k}^2}{t^2 (\sigma_n^\A)^2} \right)^\frac 1 2.
\end{align*}
The fourth moment of $U_{n,k}$ can be bounded as follows
\begin{align*}
\E^\A U_{n,k}^4 &= \frac{1}{n^2} (c^T \Sigma^{-1} x_i)^4 \E^\A (\epsilon^{(k)}_1 + \dots + \varepsilon^{(k)}_{N_k})^4\\
&= \frac{1}{n^2} (c^T \Sigma^{-1} x_i)^4 \cdot ( N_k \E \varepsilon^4 + N_k (N_k-1) (\E \varepsilon^2)^2)\\
&\leq \frac{C}{n^2} (c^T \Sigma^{-1} x_i)^4 \cdot ( N_k^2 \sigma_\varepsilon^4),
\end{align*}
where the last inequality follows because $N_k \leq N_k^2$ and by sub-Gaussianity of $\varepsilon$ we have $\E \varepsilon^4 \leq C \cdot \sigma_\varepsilon^4$ for an absolute constant $C > 0$. Combined with the equality $\E^\A U_{n,k}^2 \lesssim N_k (c^T \Sigma^{-1} x_i)^2 \sigma_\varepsilon^2$, we find the upper bound

\begin{align*}
\E^\A [ U_{n,k}^2 \mathbf{1} \{ |U_{n,k}| > t \sigma_n^{\A}  \} ] &\leq \left( \frac{C}{t^2 (\sigma_n^\A)^2} \left( \frac{ N_k}{n}\right)^3 \cdot (c^T \Sigma^{-1} x_k)^6 \right)^\frac 1 2 \\
&\leq  \frac{1}{ \sigma_n^{\A}}  \left( \frac{C}{t^2} \frac{ \sum_{i=1}^N N_i}{n} \right)^\frac 1 2 \left( \frac{N_k}{n} \right) |c^T \Sigma^{-1} x_k |^3.
\end{align*}

Therefore the expression in the Lindeberg condition has the following upper bound:
\begin{align*}
T_n \leq   \left( \frac{C}{t^2} \frac{ \sum_{i=1}^N N_i}{n} \right)^\frac 1 2 \left( \frac{(c^T \Sigma^{-1} c)}{(\sigma_n^\A)^2} \right)^\frac 3 2  \frac{1}{(c^T \Sigma^{-1} c)^\frac 3 2} \sum_{k=1}^N |c^T \Sigma^{-1} x_k |^3 \hat{w}_k
 \end{align*}
 
Since, as shown before, $\sum_i N_i / n \rightarrow_p 1$ and $(\sigma_n^\A)^2 /  (c^T \Sigma^{-1} c) \rightarrow_p 1$, it suffices to show that 

\begin{align*}
\frac{1}{(c^T \Sigma^{-1} c)^\frac 3 2} \sum_{k=1}^N |c^T \Sigma^{-1} x_k |^3 \hat{w}_k \rightarrow_p 0.
\end{align*}
We will show the variance of this term converges to zero:
\begin{align*}
\frac{1}{n \cdot (c^T \Sigma^{-1} c)^3} \sum_{k=1}^N |c^T \Sigma^{-1} x_k |^6 w_k \rightarrow 0.
\end{align*}
A similar technique as before is applicable here. Let $d_k = \Sigma^{-\frac 1 2} x_k$ and $v = \Sigma^{-\frac 1 2} c$. The variance can be rewritten as
\begin{align*}
\frac{1}{n \| v\|_2^6} \sum_k w_k (d_k^T v)^6 &\leq \frac{1}{n} \cdot \max_{\|u\|_2^2 = 1}\left\{ \sum_k w_k (d_k^T u)^6 \right\} \\
&\leq \frac{1}{n} \rightarrow 0, \quad \text{ as } n \rightarrow \infty.
\end{align*}
where the last inequality follows because the value of the maximization problem is seen to be $1$ using a similar argument as the one used in Part 2 (Variance Approximation).

\begin{proposition}\label{prop:SDP}
Problem $P2$ can be recast as a semidefinite program (SDP).
\end{proposition}
\textbf{Proof.} We can write problem $\mathbf{P2}$ as
\begin{align*}
\mathbf{P2^\prime : }\min_{t\in \R, w \in \R^N} \quad & t\\
\text{s.t.} \quad 
&\Sigma = \sum_{i=1}^N w_i x_i x_i^T, \quad \sum_{i=1}^N w_i = 1\\
& \lambda_\star I \preccurlyeq \Sigma \preccurlyeq \lambda^\star I, \quad  c^T\Sigma^{-1}c \leq t\\
&w \geq 0.
\end{align*}
The constraint $c^T\Sigma^{-1}c \leq t$ is equivalent to 
\begin{align*}
\begin{bmatrix}
t & c^T \\
c & \Sigma
\end{bmatrix} \succcurlyeq 0
\end{align*}
since, given that $\Sigma$ is positive definite (guaranteed by the constraint $\Sigma - \alpha I \succcurlyeq 0$), the above matrix is positive semidefinite if and only if the Schur complement $t - c^T\Sigma^{-1}c$ is positive semidefinite, by the following decomposition:
\begin{align*}
\begin{bmatrix}
t & c^T \\
c & \Sigma
\end{bmatrix} = 
\begin{bmatrix}
1 & c^T\Sigma^{-1} \\
0 & I
\end{bmatrix} \cdot 
\begin{bmatrix}
t - c^T\Sigma^{-1}c & 0 \\
0 & \Sigma
\end{bmatrix} \cdot
\begin{bmatrix}
1 & c^T\Sigma^{-1} \\
0 & I
\end{bmatrix}^T.
\end{align*}

\section{Appendix}\label{Appendix}
 The following provides the details left out in example 1:\newline \newline
\textbf{Example 1 (continued). } First we show that on event $E$, the lasso estimate with the theoretical value of the tuning parameter $\lambda = \sqrt{(2+\eta)\log(p) / n}$ for some $\eta > 0$ vanishes with high probability. Let $L(\beta)$ be the objective of the weighted lasso defined in equation (\ref{lasso_def}) and $\widehat{W}$ be a diagonal matrix with $\widehat{W}_j$'s on its diagonal. For any $\beta$, on $E$ we have
\begin{align*}
L(\beta) - L(0_p) &= \frac{1}{2n} \| \varepsilon - \mat{X} \beta \|_2^2 + \lambda \sum_{j = 1}^p \widehat{W}_j |\beta_j| - \| \varepsilon \|_2^2\\
&= \| \mat{X} \beta \|_2^2 + \frac{-1}{n} \varepsilon^T \mat{X} \beta + \lambda \| \widehat{W} \beta \|_1 \\
&\geq \| \mat{X} \beta \|_2^2 +  \left( \lambda - \frac{\left\| \varepsilon^T X W^+\right\|_\infty }{n} \right) \cdot \| \widehat{W} \beta \|_1.
\end{align*}
A standard union argument shows that we have
\begin{align*}
\p( \lambda > \| \varepsilon^T \mat{X} W^+ \|_\infty / n) \rightarrow 1.
\end{align*}
This implies that with probability $ 1 - o(1)$ and for any $\beta$, we have $L(\beta) \geq L(0)$, so that $0_p \in \arg\min_\beta L(\beta)$. In fact, for all $j \in [p]$ we have $\widehat{W}_j \hat{\beta}_j = 0$.
\newline
\newline
 Next, we sketch an explicit construction of the experimental domain $(x_i)_1^N$ used in Example 1. We start with the matrix $D$ of the discrete cosine transform (DCT) defined by equation (\ref{def:DCT}). The matrix $B = \sqrt{N} D^T$ satisfies $B^T B = N\cdot I_p$ and $B_{i1} = 1$ for all $1 \leq i \leq N$ and $\max_{i,j} |B_{ij}| \leq \sqrt{2}$. Denote by $B_i$ the $i$-th row of $B$ and define
\begin{align*}
x_i = \begin{cases}
\frac{1}{\sqrt{2}} ( B_i + B_{i+1}) & : i \text{ is odd},\\
\frac{1}{\sqrt{2}} ( B_{i-1} - B_{i}) & : i \text{ is even}.
\end{cases}
\end{align*}
Then it is straightforward to check that $N^{-1} \sum_i x_i x_i^T = I_p$ and $\| x_i \|_\infty \leq 2$ for all $i \leq N$. Furthermore, we have 
\begin{align*}
x_{i1} = \begin{cases}
\sqrt{2} &: i \text{ is odd},\\
0 & :  i \text{ is even}.
\end{cases}
\end{align*}
\newline
\newline
\textbf{Constructing $f_i$'s } To construct the $f_i$'s alluded to in the proof of the second part of theorem \ref{thm:debiasing}, consider the following matrix (with the $f_i$'s to be specified shortly)
\begin{align*}
    D = \begin{pmatrix}
        d_1^T & f_1^T \\
        \vdots & \vdots \\
        d_N^T & f_N^T
    \end{pmatrix} = 
    \begin{pmatrix}
    \mid & \hdots & \mid\\
    g_1 & \hdots & g_N\\
    \mid & \hdots & \mid
    \end{pmatrix}
    \in R^{N \times N}.
\end{align*}
In our notation we write $d_i = (d_{ij})_{j=1}^p$ and $g_i = (g_ij)_{j=1}^N$, so that $d_{ij}$ is the $j$-th coordinate of $d_i$, etc. Then we know that  
\begin{align*}
\langle g_i, g_j \rangle_w := \sum_k w_k g_{ik}g_{jk} = \sum_k w_k d_{ki} d_{kj} = \left[\sum_k w_k d_k d_k^T \right]_{ij} = \delta_{ij}
\end{align*}
for $1 \leq i,j \leq p$. In other words, $\{g_1, \dots, g_p\}$ forms an orthonormal basis (w.r.t. $\langle \cdot, \cdot \rangle_w$) of its span. All we need is to choose $g_{p+1}, \dots, g_{N}$ in such a way that $\{g_1, \dots, g_N\}$ is an orthonormal basis of $R^N$ under $\langle \cdot, \cdot \rangle_w$, which is easy to construct using e.g. the Gram-Schmidt procedure.  This will ensure that $\langle g_i, g_j \rangle_w = \delta_{ij}$ for all $1 \leq i,j \leq N$.

With this choice, the $(i,j)$-th coordinate of $\sum_k w_k \tilde{d}_k \tilde{d}_k^T$ is given by
\begin{align*}
     \sum_k w_k [\tilde{d}_k \tilde{d}_k^T]_{ij} &= \sum_{k} w_k \tilde{d}_{ki} \tilde{d}_{kj} \\
     &= \sum_k w_k g_{ik} g_{jk}\\
     &= \langle g_i, g_j \rangle_w = \delta_{ij} \quad \text{ for all } 1 \leq i,j \leq N.
\end{align*}
Thus $\sum_k w_k \tilde{d}_k \tilde{d}_k^T = \mathbf{I}_N$.

\bibliographystyle{imsart-nameyear} 
\bibliography{biblio.bib}       

\begin{thebibliography}{34}

\bibitem[\protect\citeauthoryear{Ahmed, Natarajan and
  Rao}{1974}]{ahmed1974discrete}
\begin{barticle}[author]
\bauthor{\bsnm{Ahmed},~\bfnm{Nasir}\binits{N.}},
  \bauthor{\bsnm{Natarajan},~\bfnm{T\_}\binits{T.}} \AND
  \bauthor{\bsnm{Rao},~\bfnm{Kamisetty~R}\binits{K.~R.}}
(\byear{1974}).
\btitle{Discrete cosine transform}.
\bjournal{IEEE transactions on Computers}
\bvolume{100}
\bpages{90--93}.
\end{barticle}
\endbibitem

\bibitem[\protect\citeauthoryear{Atkinson}{1996}]{atkinson1996usefulness}
\begin{barticle}[author]
\bauthor{\bsnm{Atkinson},~\bfnm{Anthony~C}\binits{A.~C.}}
(\byear{1996}).
\btitle{The usefulness of optimum experimental designs}.
\bjournal{Journal of the Royal Statistical Society: Series B (Methodological)}
\bvolume{58}
\bpages{59--76}.
\end{barticle}
\endbibitem

\bibitem[\protect\citeauthoryear{Bickel, Ritov and
  Tsybakov}{2009}]{bickel2009simultaneous}
\begin{barticle}[author]
\bauthor{\bsnm{Bickel},~\bfnm{Peter~J}\binits{P.~J.}},
  \bauthor{\bsnm{Ritov},~\bfnm{Ya’acov}\binits{Y.}} \AND
  \bauthor{\bsnm{Tsybakov},~\bfnm{Alexandre~B}\binits{A.~B.}}
(\byear{2009}).
\btitle{Simultaneous analysis of Lasso and {Dantzig} selector}.
\bjournal{The Annals of Statistics}
\bvolume{37}
\bpages{1705--1732}.
\end{barticle}
\endbibitem

\bibitem[\protect\citeauthoryear{B{\"u}hlmann and Van
  De~Geer}{2011}]{buhlmann2011statistics}
\begin{bbook}[author]
\bauthor{\bsnm{B{\"u}hlmann},~\bfnm{Peter}\binits{P.}} \AND \bauthor{\bsnm{Van
  De~Geer},~\bfnm{Sara}\binits{S.}}
(\byear{2011}).
\btitle{Statistics for high-dimensional data: methods, theory and
  applications}.
\bpublisher{Springer Science \& Business Media}.
\end{bbook}
\endbibitem

\bibitem[\protect\citeauthoryear{Bulinski}{2017}]{Bulinski2017ConditionalCL}
\begin{barticle}[author]
\bauthor{\bsnm{Bulinski},~\bfnm{Alexander}\binits{A.}}
(\byear{2017}).
\btitle{Conditional Central Limit Theorem}.
\bjournal{Theory of Probability and Its Applications}
\bvolume{61}
\bpages{613-631}.
\end{barticle}
\endbibitem

\bibitem[\protect\citeauthoryear{Cai, Cai and
  Guo}{2019}]{cai2019individualized}
\begin{barticle}[author]
\bauthor{\bsnm{Cai},~\bfnm{Tianxi}\binits{T.}},
  \bauthor{\bsnm{Cai},~\bfnm{Tony}\binits{T.}} \AND
  \bauthor{\bsnm{Guo},~\bfnm{Zijian}\binits{Z.}}
(\byear{2019}).
\btitle{Individualized treatment selection: An optimal hypothesis testing
  approach in high-dimensional models}.
\bjournal{arXiv preprint arXiv:1904.12891}.
\end{barticle}
\endbibitem

\bibitem[\protect\citeauthoryear{Cai and Guo}{2017}]{cai2017confidence}
\begin{barticle}[author]
\bauthor{\bsnm{Cai},~\bfnm{T~Tony}\binits{T.~T.}} \AND
  \bauthor{\bsnm{Guo},~\bfnm{Zijian}\binits{Z.}}
(\byear{2017}).
\btitle{Confidence intervals for high-dimensional linear regression: Minimax
  rates and adaptivity}.
\bjournal{The Annals of statistics}
\bvolume{45}
\bpages{615--646}.
\end{barticle}
\endbibitem

\bibitem[\protect\citeauthoryear{Chen, Donoho and
  Saunders}{2001}]{chen2001atomic}
\begin{barticle}[author]
\bauthor{\bsnm{Chen},~\bfnm{Scott~Shaobing}\binits{S.~S.}},
  \bauthor{\bsnm{Donoho},~\bfnm{David~L}\binits{D.~L.}} \AND
  \bauthor{\bsnm{Saunders},~\bfnm{Michael~A}\binits{M.~A.}}
(\byear{2001}).
\btitle{Atomic decomposition by basis pursuit}.
\bjournal{SIAM review}
\bvolume{43}
\bpages{129--159}.
\end{barticle}
\endbibitem

\bibitem[\protect\citeauthoryear{Chernoff}{1953}]{chernoff1953locally}
\begin{barticle}[author]
\bauthor{\bsnm{Chernoff},~\bfnm{Herman}\binits{H.}}
(\byear{1953}).
\btitle{Locally optimal designs for estimating parameters}.
\bjournal{The Annals of Mathematical Statistics}
\bpages{586--602}.
\end{barticle}
\endbibitem

\bibitem[\protect\citeauthoryear{Deng, Lin and Qian}{2013}]{Deng_thelasso}
\begin{bmisc}[author]
\bauthor{\bsnm{Deng},~\bfnm{Xinwei}\binits{X.}},
  \bauthor{\bsnm{Lin},~\bfnm{C.~Devon}\binits{C.~D.}} \AND
  \bauthor{\bsnm{Qian},~\bfnm{Peter Z.~G.}\binits{P.~Z.~G.}}
(\byear{2013}).
\btitle{The Lasso with Nearly Orthogonal Latin Hypercube Designs}.
\end{bmisc}
\endbibitem

\bibitem[\protect\citeauthoryear{Dette et~al.}{1993}]{dette1993elfving}
\begin{barticle}[author]
\bauthor{\bsnm{Dette},~\bfnm{Holger}\binits{H.}} \betal{et~al.}
(\byear{1993}).
\btitle{Elfving's Theorem for $ D $-Optimality}.
\bjournal{The Annals of Statistics}
\bvolume{21}
\bpages{753--766}.
\end{barticle}
\endbibitem

\bibitem[\protect\citeauthoryear{Dette et~al.}{2009}]{dette2009geometric}
\begin{barticle}[author]
\bauthor{\bsnm{Dette},~\bfnm{Holger}\binits{H.}},
  \bauthor{\bsnm{Holland-Letz},~\bfnm{Tim}\binits{T.}} \betal{et~al.}
(\byear{2009}).
\btitle{A geometric characterization of c-optimal designs for heteroscedastic
  regression}.
\bjournal{The Annals of Statistics}
\bvolume{37}
\bpages{4088--4103}.
\end{barticle}
\endbibitem

\bibitem[\protect\citeauthoryear{Dobriban and
  Fan}{2016}]{dobriban2016regularity}
\begin{barticle}[author]
\bauthor{\bsnm{Dobriban},~\bfnm{Edgar}\binits{E.}} \AND
  \bauthor{\bsnm{Fan},~\bfnm{Jianqing}\binits{J.}}
(\byear{2016}).
\btitle{Regularity properties for sparse regression}.
\bjournal{Communications in mathematics and statistics}
\bvolume{4}
\bpages{1--19}.
\end{barticle}
\endbibitem

\bibitem[\protect\citeauthoryear{Elfving}{1952}]{elfving1952optimum}
\begin{barticle}[author]
\bauthor{\bsnm{Elfving},~\bfnm{Gustav}\binits{G.}}
(\byear{1952}).
\btitle{Optimum allocation in linear regression theory}.
\bjournal{The Annals of Mathematical Statistics}
\bvolume{23}
\bpages{255--262}.
\end{barticle}
\endbibitem

\bibitem[\protect\citeauthoryear{Huang, Kong and Ai}{2020}]{huang2020optimal}
\begin{barticle}[author]
\bauthor{\bsnm{Huang},~\bfnm{Yimin}\binits{Y.}},
  \bauthor{\bsnm{Kong},~\bfnm{Xiangshun}\binits{X.}} \AND
  \bauthor{\bsnm{Ai},~\bfnm{Mingyao}\binits{M.}}
(\byear{2020}).
\btitle{Optimal designs in sparse linear models}.
\bjournal{Metrika}
\bvolume{83}
\bpages{255--273}.
\end{barticle}
\endbibitem

\bibitem[\protect\citeauthoryear{Javanmard and
  Lee}{2017}]{javanmard2017flexible}
\begin{barticle}[author]
\bauthor{\bsnm{Javanmard},~\bfnm{Adel}\binits{A.}} \AND
  \bauthor{\bsnm{Lee},~\bfnm{Jason~D}\binits{J.~D.}}
(\byear{2017}).
\btitle{A flexible framework for hypothesis testing in high-dimensions}.
\bjournal{arXiv preprint arXiv:1704.07971}.
\end{barticle}
\endbibitem

\bibitem[\protect\citeauthoryear{Javanmard and
  Montanari}{2014}]{javanmard2014confidence}
\begin{barticle}[author]
\bauthor{\bsnm{Javanmard},~\bfnm{Adel}\binits{A.}} \AND
  \bauthor{\bsnm{Montanari},~\bfnm{Andrea}\binits{A.}}
(\byear{2014}).
\btitle{Confidence intervals and hypothesis testing for high-dimensional
  regression}.
\bjournal{The Journal of Machine Learning Research}
\bvolume{15}
\bpages{2869--2909}.
\end{barticle}
\endbibitem

\bibitem[\protect\citeauthoryear{Kohavi et~al.}{2009}]{kohavi2009controlled}
\begin{barticle}[author]
\bauthor{\bsnm{Kohavi},~\bfnm{Ron}\binits{R.}},
  \bauthor{\bsnm{Longbotham},~\bfnm{Roger}\binits{R.}},
  \bauthor{\bsnm{Sommerfield},~\bfnm{Dan}\binits{D.}} \AND
  \bauthor{\bsnm{Henne},~\bfnm{Randal~M}\binits{R.~M.}}
(\byear{2009}).
\btitle{Controlled experiments on the web: survey and practical guide}.
\bjournal{Data mining and knowledge discovery}
\bvolume{18}
\bpages{140--181}.
\end{barticle}
\endbibitem

\bibitem[\protect\citeauthoryear{Lustig, Donoho and
  Pauly}{2007}]{lustig2007sparse}
\begin{barticle}[author]
\bauthor{\bsnm{Lustig},~\bfnm{Michael}\binits{M.}},
  \bauthor{\bsnm{Donoho},~\bfnm{David}\binits{D.}} \AND
  \bauthor{\bsnm{Pauly},~\bfnm{John~M}\binits{J.~M.}}
(\byear{2007}).
\btitle{Sparse MRI: The application of compressed sensing for rapid MR
  imaging}.
\bjournal{Magnetic Resonance in Medicine: An Official Journal of the
  International Society for Magnetic Resonance in Medicine}
\bvolume{58}
\bpages{1182--1195}.
\end{barticle}
\endbibitem

\bibitem[\protect\citeauthoryear{Misra and Wu}{2020}]{MISRA2020289}
\begin{bbook}[author]
\bauthor{\bsnm{Misra},~\bfnm{Siddharth}\binits{S.}} \AND
  \bauthor{\bsnm{Wu},~\bfnm{Yaokun}\binits{Y.}}
(\byear{2020}).
\btitle{Machine learning assisted segmentation of scanning electron microscopy
  images of organic-rich shales with feature extraction and feature ranking}.
\bpublisher{Gulf Professional Publishing}.
\bdoi{https://doi.org/10.1016/B978-0-12-817736-5.00010-7}
\end{bbook}
\endbibitem

\bibitem[\protect\citeauthoryear{Pukelsheim}{2006}]{pukelsheim2006optimal}
\begin{bbook}[author]
\bauthor{\bsnm{Pukelsheim},~\bfnm{Friedrich}\binits{F.}}
(\byear{2006}).
\btitle{Optimal design of experiments}.
\bpublisher{SIAM}.
\end{bbook}
\endbibitem

\bibitem[\protect\citeauthoryear{Pukelsheim and
  Rieder}{1992}]{pukelsheim1992efficient}
\begin{barticle}[author]
\bauthor{\bsnm{Pukelsheim},~\bfnm{Friedrich}\binits{F.}} \AND
  \bauthor{\bsnm{Rieder},~\bfnm{Sabine}\binits{S.}}
(\byear{1992}).
\btitle{Efficient rounding of approximate designs}.
\bjournal{Biometrika}
\bvolume{79}
\bpages{763--770}.
\end{barticle}
\endbibitem

\bibitem[\protect\citeauthoryear{Ravi et~al.}{2016}]{ravi2016experimental}
\begin{binproceedings}[author]
\bauthor{\bsnm{Ravi},~\bfnm{Sathya~Narayanan}\binits{S.~N.}},
  \bauthor{\bsnm{Ithapu},~\bfnm{Vamsi}\binits{V.}},
  \bauthor{\bsnm{Johnson},~\bfnm{Sterling}\binits{S.}} \AND
  \bauthor{\bsnm{Singh},~\bfnm{Vikas}\binits{V.}}
(\byear{2016}).
\btitle{Experimental design on a budget for sparse linear models and
  applications}.
In \bbooktitle{International Conference on Machine Learning}
\bpages{583--592}.
\end{binproceedings}
\endbibitem

\bibitem[\protect\citeauthoryear{Rodr{\'\i}guez-D{\'\i}az}{2017}]{rodriguez2017computation}
\begin{barticle}[author]
\bauthor{\bsnm{Rodr{\'\i}guez-D{\'\i}az},~\bfnm{Juan~M}\binits{J.~M.}}
(\byear{2017}).
\btitle{Computation of c-optimal designs for models with correlated
  observations}.
\bjournal{Computational Statistics \& Data Analysis}
\bvolume{113}
\bpages{287--296}.
\end{barticle}
\endbibitem

\bibitem[\protect\citeauthoryear{Rudelson and
  Zhou}{2012}]{rudelson2012reconstruction}
\begin{binproceedings}[author]
\bauthor{\bsnm{Rudelson},~\bfnm{Mark}\binits{M.}} \AND
  \bauthor{\bsnm{Zhou},~\bfnm{Shuheng}\binits{S.}}
(\byear{2012}).
\btitle{Reconstruction from anisotropic random measurements}.
In \bbooktitle{Conference on Learning Theory}
\bpages{10--1}.
\end{binproceedings}
\endbibitem

\bibitem[\protect\citeauthoryear{Rudin, Osher and
  Fatemi}{1992}]{rudin1992nonlinear}
\begin{barticle}[author]
\bauthor{\bsnm{Rudin},~\bfnm{Leonid~I}\binits{L.~I.}},
  \bauthor{\bsnm{Osher},~\bfnm{Stanley}\binits{S.}} \AND
  \bauthor{\bsnm{Fatemi},~\bfnm{Emad}\binits{E.}}
(\byear{1992}).
\btitle{Nonlinear total variation based noise removal algorithms}.
\bjournal{Physica D: nonlinear phenomena}
\bvolume{60}
\bpages{259--268}.
\end{barticle}
\endbibitem

\bibitem[\protect\citeauthoryear{Sagnol}{2011}]{sagnol2011computing}
\begin{barticle}[author]
\bauthor{\bsnm{Sagnol},~\bfnm{Guillaume}\binits{G.}}
(\byear{2011}).
\btitle{Computing optimal designs of multiresponse experiments reduces to
  second-order cone programming}.
\bjournal{Journal of Statistical Planning and Inference}
\bvolume{141}
\bpages{1684--1708}.
\end{barticle}
\endbibitem

\bibitem[\protect\citeauthoryear{Seeger}{2008}]{seeger2008bayesian}
\begin{barticle}[author]
\bauthor{\bsnm{Seeger},~\bfnm{Matthias~W}\binits{M.~W.}}
(\byear{2008}).
\btitle{Bayesian inference and optimal design for the sparse linear model}.
\bjournal{Journal of Machine Learning Research}
\bvolume{9}
\bpages{759--813}.
\end{barticle}
\endbibitem

\bibitem[\protect\citeauthoryear{Tibshirani}{1996}]{tibshirani1996regression}
\begin{barticle}[author]
\bauthor{\bsnm{Tibshirani},~\bfnm{Robert}\binits{R.}}
(\byear{1996}).
\btitle{Regression shrinkage and selection via the lasso}.
\bjournal{Journal of the Royal Statistical Society: Series B (Methodological)}
\bvolume{58}
\bpages{267--288}.
\end{barticle}
\endbibitem

\bibitem[\protect\citeauthoryear{van~de Geer et~al.}{2014}]{vandegeer2014}
\begin{barticle}[author]
\bauthor{\bparticle{van~de} \bsnm{Geer},~\bfnm{Sara}\binits{S.}},
  \bauthor{\bsnm{Bühlmann},~\bfnm{Peter}\binits{P.}},
  \bauthor{\bsnm{Ritov},~\bfnm{Ya’acov}\binits{Y.}} \AND
  \bauthor{\bsnm{Dezeure},~\bfnm{Ruben}\binits{R.}}
(\byear{2014}).
\btitle{On asymptotically optimal confidence regions and tests for
  high-dimensional models}.
\bjournal{Ann. Statist.}
\bvolume{42}
\bpages{1166--1202}.
\bdoi{10.1214/14-AOS1221}
\end{barticle}
\endbibitem

\bibitem[\protect\citeauthoryear{{\v{C}}ern{\`y} and
  Hlad{\'\i}k}{2012}]{vcerny2012two}
\begin{barticle}[author]
\bauthor{\bsnm{{\v{C}}ern{\`y}},~\bfnm{Michal}\binits{M.}} \AND
  \bauthor{\bsnm{Hlad{\'\i}k},~\bfnm{Milan}\binits{M.}}
(\byear{2012}).
\btitle{Two complexity results on c-optimality in experimental design}.
\bjournal{Computational Optimization and Applications}
\bvolume{51}
\bpages{1397--1408}.
\end{barticle}
\endbibitem

\bibitem[\protect\citeauthoryear{Vershynin}{2018}]{vershynin2018high}
\begin{bbook}[author]
\bauthor{\bsnm{Vershynin},~\bfnm{Roman}\binits{R.}}
(\byear{2018}).
\btitle{High-dimensional probability: An introduction with applications in data
  science}
\bvolume{47}.
\bpublisher{Cambridge University Press}.
\end{bbook}
\endbibitem

\bibitem[\protect\citeauthoryear{Weng}{2015}]{weng2015optimizing}
\begin{barticle}[author]
\bauthor{\bsnm{Weng},~\bfnm{Chunhua}\binits{C.}}
(\byear{2015}).
\btitle{Optimizing clinical research participant selection with informatics}.
\bjournal{Trends in pharmacological sciences}
\bvolume{36}
\bpages{706--709}.
\end{barticle}
\endbibitem

\bibitem[\protect\citeauthoryear{Zhang and Zhang}{2014}]{zhang2014confidence}
\begin{barticle}[author]
\bauthor{\bsnm{Zhang},~\bfnm{Cun-Hui}\binits{C.-H.}} \AND
  \bauthor{\bsnm{Zhang},~\bfnm{Stephanie~S}\binits{S.~S.}}
(\byear{2014}).
\btitle{Confidence intervals for low dimensional parameters in high dimensional
  linear models}.
\bjournal{Journal of the Royal Statistical Society: Series B (Statistical
  Methodology)}
\bvolume{76}
\bpages{217--242}.
\end{barticle}
\endbibitem

\end{thebibliography}


\end{document}